\documentclass{amsart}
\usepackage[utf8]{inputenc}
\usepackage{amsmath}
\usepackage{amsfonts}
\usepackage{amssymb}
\usepackage{graphicx}
\usepackage{bbm}
\usepackage{mathrsfs}
\usepackage[all]{xy}
\usepackage{manfnt}
\usepackage[backref=page]{hyperref}
\usepackage[left=2cm,right=2cm,top=2cm,bottom=2cm]{geometry}
\usepackage[color=blue!30,textsize=tiny]{todonotes}

\author{Geoffrey Powell}
\title[Functors on finite sets]{Functors on the category of finite sets  revisited}
\address{Univ Angers, CNRS, LAREMA, SFR MATHSTIC, F-49000 Angers, France}
\email{Geoffrey.Powell@math.cnrs.fr}
\urladdr{https://math.univ-angers.fr/~powell/}


\newtheorem{thm}{Theorem}[section]
\newtheorem{proposition}[thm]{Proposition}
\newtheorem{corollary}[thm]{Corollary}
\newtheorem{lemma}[thm]{Lemma}
\theoremstyle{definition}
\newtheorem{definition}[thm]{Definition}
\newtheorem{example}[thm]{Example}

\theoremstyle{remark}
\newtheorem{remark}[thm]{Remark}
\newtheorem*{remark*}{Remark}
\newtheorem{notation}[thm]{Notation}

\newtheorem{conv}[thm]{Convention}

\newcommand{\f}{\mathcal{F}}
\renewcommand{\phi}{\varphi}
\renewcommand{\hom}{\mathrm{Hom}}
\newcommand{\sym}{\mathfrak{S}}
\newcommand{\fs}{{\mathbf{FS}}}
\newcommand{\kmod}{\mathtt{Mod}_\kring}
\newcommand{\calc}{\mathcal{C}}
\newcommand{\nat}{\mathbb{N}}
\newcommand{\zed}{\mathbb{Z}}

\newcommand{\op}{^\mathrm{op}}
\newcommand{\ob}{\mathrm{Ob}\hspace{2pt}}
\newcommand{\kring}{\mathbbm{k}}
\newcommand{\fb}{\mathbf{FB}}
\newcommand{\finj}{{\mathbf{FI}}}
\newcommand{\id}{\mathrm{Id}}
\newcommand{\pbar}{\overline{P}}
\newcommand{\filt}{\mathfrak{f}}
\newcommand{\triv}{\mathsf{triv}}
\newcommand{\sgn}{\mathsf{sgn}}
\newcommand{\aut}{\mathrm{Aut}}
\newcommand{\fin}{\mathbf{FA}}
\newcommand{\hs}{\mathrm{HS}}
\newcommand{\pfin}{P^\fin}
\newcommand{\n}{\mathbf{n}}
\newcommand{\m}{\mathbf{m}}
\newcommand{\kbar}{\overline{\kring}}

\newcommand{\W}{\mathcal{S}}

\newcommand{\bse}[1]{\backslash \{ #1 \}}
\newcommand{\bs}{\backslash}
\newcommand{\llproj}{\mathbf{Q}}
\newcommand{\llop}{\mathbf{R}}
\numberwithin{equation}{section}
\pdfmapline{=eufm7 eufm7 <eufm7.pfb}
\pdfmapline{=eufm10 eufm10 <eufm10.pfb}

\begin{document}

\begin{abstract}
We study the structure of the category of representations of $\mathbf{FA}$, the category of finite sets and all maps,  working over a unital commutative ring. We first exhibit a convenient set of projective generators by exploiting Möbius inversion and establish some general structural results. 

We then specialize to working over a field of characteristic zero. We  construct the simple representations explicitly, together with their projective covers. This recovers  the classification of the simples given by Wiltshire-Gordon.  

Hence we obtain a more efficient set of projective generators of the category of representations of $\mathbf{FA}$ and the associated  Morita equivalence. This is used to explain how to calculate the multiplicities of the composition factors of an arbitrary object, based only on its underlying $\mathbf{FB}$-representation, where $\mathbf{FB}$ is the category of finite sets and bijections. 

These results are applied to show how to calculate the morphism spaces between projectives in our chosen set of generators, as well as for a closely related family of objects. 
\end{abstract}

\maketitle

\section{Introduction}
\label{sect:intro}

Recently, the study of representations of categories of a combinatorial nature has been of major interest. (Here a representation of an essentially small category $\calc$ is understood to be  a functor from $\calc$ to $\kmod$, the category of $\kring$-modules over a unital commutative ring $\kring$; these form a category denoted $\f (\calc)$.) This is especially true for categories related to $\fin$, the category of finite sets, including the subcategory $\finj$ of finite sets and injections, the subcategory $\fs$ of finite sets and surjections, and $\fb$ that of bijections (see \cite{MR3556290} for example). Here our main interest is in  $\f (\fin)$, also referred to here as $\kring \fin$-modules.

This paper starts by establishing some results that hold without supposing that $\kring$ is a field. One has the family of standard projective generators $\pfin_\n$, for $n \in \nat$, where $\pfin_\n(X) = \kring [X]^{\otimes n}$, for a finite set $X$; $\n$ denotes the set $\{ 1, \ldots , n \}$. The functor $\pbar$ is defined by the short exact sequence 
\begin{eqnarray}
\label{eqn:pbar_ses_INTRO}
0
\rightarrow \pbar \rightarrow 
\pfin_\mathbf{1}
\rightarrow 
\kbar 
\rightarrow 
0,
\end{eqnarray}
where $\kbar$ is supported on the non-empty finite sets with constant value $\kring$ and the surjection  $\pfin_\mathbf{1}
\twoheadrightarrow 
\kbar$ is induced by the map $X \rightarrow \{*\}$ to the terminal object of $\fin$. On forming the $n$-fold tensor product, for $n \in \nat$, one has the inclusion $\pbar^{\otimes n} \hookrightarrow \pfin_\n$ that is $\sym_n\op$-equivariant; by convention, $\pbar^{\otimes 0}$ is taken to be $\kbar$. 

The functors $\pbar^{\otimes n}$ are not projective but are not far from being so. Indeed, the functor $\pbar^{\otimes n} \otimes \pfin_\mathbf{1}$ is projective (see Proposition \ref{prop:splitting_pfin_1}) and this sits in a short exact sequence 
\[
0 
\rightarrow 
\pbar^{\otimes n+1} 
\rightarrow 
\pbar^{\otimes n} \otimes \pfin_\mathbf{1}
\rightarrow 
\pbar^{\otimes n}
\rightarrow 
0.
\]

Moreover, the family $\{ \pbar^{\otimes n} \otimes \pfin_\mathbf{1} \mid n \in \nat\} \amalg \{ \pfin_\mathbf{0} = \kring \}$ gives a  more efficient set of projective generators of $\f (\fin)$. This is proved in Section \ref{sect:proj_bis} by using the general framework of Möbius inversion, thus giving a conceptual explanation of this result (see Corollary \ref{cor:llproj_projective_generators}). 

In Section \ref{sect:proj_bis}, for $(X,x)$ a finite pointed set, we work with an explicit direct summand $\llproj_\emptyset^{X,x}$ of $\pfin_X$, which is shown to be isomorphic to $\pfin_\mathbf{1} \otimes \pbar^{\otimes X \bse{x}}$. The explicit construction allows some control of morphisms between such objects; in particular we prove the fundamental result: 

\begin{thm}
(Theorem \ref{thm:morphisms_llproj}.)
For finite pointed sets $(X,x)$ and $(Y,y)$, there are natural isomorphisms:
\[
\hom_{\kring \fin} (\llproj^{X,x}_\emptyset, \llproj^{Y,y}_\emptyset)
\cong 
\left\{
\begin{array}{ll}
0 & |X|\geq |Y|+2 \\
\kring \fb (Y, X\bse{x}) & |X|= |Y|+1.
\end{array}
\right.
\]
\end{thm}

This then yields some information on morphisms between functors of the form $\pbar^{\otimes n}$: 

\begin{corollary}
(Corollary \ref{cor:morphisms_pbar}.)
For $s, t \in \nat$, if $s>t$ then 
$$
\hom_{\kring \fin} (\pbar^{\otimes s} , \pbar^{\otimes t}) =0. 
$$  
 For $s=t$,  the place permutation action on $\pbar ^{\otimes t}$ induces an isomorphism  of $\kring$-algebras:
$$
\kring \sym_t
\cong 
\mathrm{End}_{\kring \fin} (\pbar ^{\otimes t}).$$
\end{corollary}

One consequence of these results is  that, when $\kring$ is  a field, the indecomposable projectives in $\f (\fin)$ admit (finite length) composition series (see Corollary \ref{cor:finite_length_indec_projectives}).

Now, for each $n \in \nat$, there is a natural surjection of $\kring \fin\op$-modules
\begin{eqnarray}
\label{eqn:surj_fin_fs}
\kring \fin (-, \n) 
\twoheadrightarrow 
\kring \fs (-,\n),
\end{eqnarray}
where $\fs$ is the wide subcategory of $\fin$ given by the surjective maps; this surjection is defined by sending non-surjective maps to zero. Using the analysis of the projectives $\pbar^{\otimes n} \otimes \pfin_\mathbf{1}$, one has (see Corollary \ref{cor:model_right_augmentation}):

\begin{thm}
\label{thm:pbar_fs}
For $n \in \nat$, the map 
\begin{eqnarray*}
\hom_{\kring \fin} (\pfin_\n, \pfin_\bullet) 
\rightarrow 
\hom_{\kring \fin} (\pbar^{\otimes n}, \pfin_\bullet)
\end{eqnarray*}
induced by  $\pbar^{\otimes n} \hookrightarrow \pfin_\n$ identifies with (\ref{eqn:surj_fin_fs}).
 In particular, there is an isomorphism $\hom_{\kring \fin} (\pbar^{\otimes n}, \pfin_\bullet) \cong \kring \fs (\bullet , \n)$.
\end{thm}

We then turn to the case where $\kring$ is a field of characteristic zero. In this case, the simple $\kring \fin$-modules were classified by Wiltshire-Gordon \cite{2014arXiv1406.0786W}. In  \cite{2022arXiv221000399S}, Sam, Snowden and Tosteson showed that the category of $\kring \fin\op$-modules is equivalent to the category of polynomial representations of the Witt Lie algebra. Using this, they observed that Wiltshire-Gordon's simple modules correspond to a family of representations of the Witt Lie algebra constructed by Feigin and Shoikhet \cite{MR2350124}. Part of the paper revisits this classification using our more efficient set of projective generators. 

We describe the simple $\kring \fin$-modules explicitly, together with their projective covers, based on the family of  surjections (for $n \in \nat$)
 \[
\pfin_\n \twoheadrightarrow \kring \finj (\n, -)
\]
sending non-injective maps to zero, analogous to (\ref{eqn:surj_fin_fs}). This is a morphism of $\kring \sym_n\op \otimes \kring \fin$-modules. Now, for $n$ a positive integer and $\lambda \vdash n$ a partition such that $\lambda \neq (1^n)$, with associated simple $\sym_n$-module $S_\lambda$ (which is thus not isomorphic to the sign representation $\sgn_n$), one has the $\kring \fin$-module:
\[
\tilde{C}_\lambda : = \kring \finj (\n, -) \otimes _{\sym_n} S_\lambda.
\]
One also has  $S_\lambda (\pbar) := \pbar^{\otimes n} \otimes_{\sym_n} S_\lambda$, equipped with the induced map $S_\lambda (\pbar) \rightarrow \tilde{C}_\lambda$.   

For the partitions of the form $(1^n)$, one defines:
\[
\tilde{C}_{(1^n)} := \Lambda^{n-1} (\pbar).
\]
Here we use the convention that $\tilde{C}_{(1^0)} = \kring _\mathbf{0}$, the functor supported on $\emptyset$ with value $\kring$, and $\tilde{C}_{(1)} = \kbar$. There is a surjection $\Lambda^{n+1} (\pfin_\mathbf{1}) \twoheadrightarrow \Lambda^n (\pbar)$.

The classification of the simple functors is as follows:

\begin{thm}
(Theorem \ref{thm:classify_simples}.)
For $n \in \nat$ and $\lambda \vdash n$, the $\kring\fin$-module $\tilde{C}_\lambda$ is simple. Moreover, 
\begin{enumerate}
\item 
if $\lambda \neq (1^n)$, $S_\lambda (\pbar) \rightarrow \tilde{C}_\lambda$ is the projective cover of $\tilde{C}_\lambda$; 
\item 
$\Lambda^n (\pfin) \twoheadrightarrow \Lambda^{n-1}(\pbar) = \tilde{C}_{(1^n)}$ is the projective cover of $\tilde{C}_{(1^n)}$.
\end{enumerate}

For $F$  a simple $\kring\fin$-module, such that $F(\mathbf{t})=0$ for $t <n $  and  $F(\n)\neq 0$,
  $F(\n)$ is a simple $\sym_n$-module and 
$
F \cong \tilde{C}_\lambda
$
 where $F(\n) \cong S_\lambda$. (If $F(\mathbf{0})\neq 0$, this corresponds to $F \cong \kring_\mathbf{0}$.)
\end{thm}

This recovers Wiltshire-Gordon's classification  \cite{2014arXiv1406.0786W}, giving an explicit description of the simples and also of their projective covers.  This construction  of the  $\tilde{C}_\lambda$ also establishes the relation with the work of Feigin and Shoikhet \cite{MR2350124}, which can be interpreted (via the results of \cite{2022arXiv221000399S}) as being equivalent to the analysis of $\kring \finj (\n , -)$ as a $\kring \fin$-module. The latter is almost semi-simple:

\begin{corollary}
(Corollary \ref{cor:structure_pfinj}.)
For $0< n \in \nat$, there is an isomorphism of $\kring\fin$-modules:
\[
\kring \finj (\n, -) 
\cong 
\Lambda^n (\pfin_\mathbf{1}) \oplus \bigoplus_{\substack{\lambda \vdash n \\ \lambda \neq (1^n) }}
\tilde{C}_\lambda^{\oplus \dim S_\lambda}.
\]
\end{corollary}

Now, whilst $\pbar^{\otimes n}$ is  not projective,   it is very close to being so.  More precisely, writing $\mathbb{P}_n$ for the projective cover of $\pbar^{\otimes n}$, one has a $\sym_n\op$-equivariant exact sequence 
\begin{eqnarray}
\label{eqn:ses_proj_cover_INTRO}
0
\rightarrow 
\Lambda^{n+1} (\pbar ) \boxtimes \sgn_n 
\rightarrow 
\mathbb{P}_n 
\rightarrow 
\pbar^{\otimes n} 
\rightarrow 
0,
\end{eqnarray}
in which  $\Lambda^{n+1} (\pbar ) \boxtimes \sgn_n$ is simple. The explicit calculations of this paper rely upon  controlling the `error terms' that arise from the non-projectivity of $\Lambda^n (\pbar) $.

\begin{remark}
\label{rem:Dkfs_inj}
For $n \in \nat$, there is a counterpart for the structure of the  $\kring \fin$-module $D \kring \fs (-,\n)$, where $D$ is the duality functor $\f (\fin\op) \op\rightarrow \f (\fin)$ induced by vector space duality. Namely, Theorem \ref{thm:pbar_fs} together with the fact that $\pbar^{\otimes n}$ is `almost' a direct summand of $\pfin_\n$  imply that $D \kring \fs (-,\n)$ is `almost' an injective $\kring\fin$-module. The failure to be injective is encoded by the dual of the short exact sequence given in Proposition \ref{prop:calculate_sgn_otimes_Pfin}. (These observations are not developed further in this paper.)   
\end{remark}

If one defines $\mathbb{P}_{-1} := \pfin_\mathbf{0}$, then  $\{ \mathbb{P}_n \mid -1 \leq n \}$ is a set of projective generators of the category $\f (\fin)$. This yields a Morita equivalence as follows, writing $\calc (\mathbb{P}_\bullet)$ for the full subcategory of $\f (\fin)$ on these projectives:

\begin{thm}
\label{THM:Morita}
(Theorem \ref{thm:fin-modules_morita})
The category  $\f (\fin)$ is equivalent to the category of $\calc(\mathbb{P}_\bullet)\op$-modules.
\end{thm}

It follows that, to calculate the composition factors of a given $\kring \fin$-module $F$, it suffices to calculate 
$ 
\hom_{\kring \fin} (\mathbb{P}_\bullet, F)
 $ 
as a $\kring \fb$-module (see Corollary \ref{cor:multiplicity_composition_factors}).  Restricting to 
$\{ \mathbb{P}_n \mid n \in \nat \}$, one has  the following, in which, for a natural number $t$,  $\W(t) = \sum_{\ell \geq 0} (-1)^{t+\ell} [\sgn_{t+\ell}] $ (working in the Grothendieck group of $\kring \fb$-modules) and $\odot_\fb$ denotes the Day convolution product.

\begin{thm}
(Theorem \ref{thm:hom_proj_cover})
For a $\kring \fin$-module $F$ such that $\dim F(\mathbf{t})< \infty$ for all $t \in \nat$, there is an equality in the Grothendieck group of $\kring \fb$-modules
\[
[\hom_{\kring \fin}  (\mathbb{P}_\bullet, F)]
=
[\overline{F}]\odot_\fb \W(0)
\quad + \quad
\sum_{k\geq 1}
(\sgn_k \otimes_{\sym_k} F(\mathbf{k})) \otimes \W (k-1), 
\]
where $\overline{F}$ is the sub $\kring \fin$-module of $F$ supported on non-empty finite sets.
\end{thm}

These results are applied to consider the cases of most immediate interest to us, $F \in \{ \pfin_\n, \pbar^{\otimes n}, \mathbb{P}_n \mid n \in \nat \}$ in the following results:

\begin{enumerate}
\item 
Theorem \ref{thm:hom_proj_cover_pfin} determines  
$\hom_{\kring \fin} (\mathbb{P}_\bullet, \pfin_\n)$, thus giving an answer to the `urgent question' posed by Wiltshire-Gordon in \cite[Section 5.2]{2014arXiv1406.0786W} asking how to calculate the composition factors of the  projectives  $S_\lambda\otimes_{\sym_n}\pfin_\n$, for partitions $\lambda \vdash n$. 
\item 
Theorem \ref{thm:hom_proj_cover_pbar_otimes} determines $\hom_{\kring \fin} (\mathbb{P}_\bullet, \pbar^{\otimes n})$, which yields the  composition factors of the  projectives  $S_\lambda\otimes_{\sym_n}\pbar^{\otimes n}$, for partitions $\lambda \vdash n$ with $\lambda \neq (1^n)$. 
\item 
Corollary \ref{cor:endo_proj_cover} determines
$\hom_{\kring \fin}(\mathbb{P}_\bullet, \mathbb{P}_n ) $, an essential  step towards  {\em applying} the Morita equivalence result, Theorem \ref{THM:Morita}.   
\end{enumerate}

For instance, Theorem \ref{thm:hom_proj_cover_pfin} is the following result:

\begin{thm}
For $n \in \nat$, there is an  equality in the Grothendieck group of $ \kring \sym_n\op \otimes \kring \fb $-modules:
\[
[\hom_{\kring \fin}  (\mathbb{P}_\bullet, \pfin_\n)]
=
[\kring \fs (\n, -)]  
\quad + \quad
\sum_{k\geq 1}
(\sgn_k \otimes_{\sym_k} \kring \fs (\n , \mathbf{k})) \otimes [\sgn_{k-1}].
\]
\end{thm}

These results also allow the calculation of morphisms between functors of the form $\pbar^{\otimes n}$:

\begin{thm}
\label{THM:endo_pbar}
(Theorem \ref{thm:endo_pbar_otimes})
There is an  equality in the Grothendieck group of $\kring \fb$-bimodules:
\begin{eqnarray*} 
[\hom_{\kring \fin}  (\pbar^{\otimes \bullet}, \pbar^{\otimes *})]
&=&
[\kring \fs ]\odot_{\fb\op} \W(0) 
\  + \ 
\sum_{k \geq 1} \W (k) \boxtimes [\sgn_{k-1}]_\fb
\\
&=&
[\kring \fs ]\odot_{\fb\op} \W(0) 
\  + \ 
\sum_{k \geq 1}\sum_{t \in \nat} (-1)^t [\sgn_{k+t}]_{\fb\op}\boxtimes [\sgn_{k-1}]_\fb.
\end{eqnarray*}
\end{thm}

Theorem \ref{THM:endo_pbar} recovers, by very different methods, one of the results of \cite{MR4518761} (see also \cite{2024arXiv240711627P}).

\medskip
\noindent
{\bf Organization:}
The paper is presented in three Parts, separating  out the main themes. 
 Part \ref{part:arbitrary} provides background and establishes some general results that hold over any unital commutative ring. 
 Part \ref{part:zero} gives the classification of the simple $\kring \fin$-modules and their projective covers, working over a field of characteristic zero; this yields the  Morita equivalence.  
Part \ref{part:calculate}  shows how to carry out associated basic calculations, including how to calculate the composition factors of a $\kring \fin$-module.

\ 

\medskip

\noindent
{\bf Acknowledgements:}
The precursor of this research was the author's work with Christine Vespa \cite{MR4518761}, which took a very different approach. The author is very grateful  for her ongoing interest. 

The author also thanks Vladimir Dotsenko for an invitation to Strasbourg in January 2024 (financed by his {\em IUF}) and  for useful discussions - during which the work of Feigin and Shoikhet appeared from a different viewpoint. Some of the research underlying this paper was carried out during that visit.

The author is also immensely grateful to the referee, both for their careful reading of the manuscript and for their suggestion that Möbius inversion gives a conceptual approach to the general results of Part \ref{part:arbitrary}, also strengthening some results, for which they provided details. This lead to the current, less {\em ad hoc} presentation of Section \ref{sect:proj_bis} and also lead the author to only impose the hypothesis that $\kring$ is a field when necessary. 

\tableofcontents

\part{Working over a unital commutative ring}
\label{part:arbitrary}

\section{Background}
\label{sect:background}

The purpose of this section is to recall some generalities on the structure of the category of $\kring\calc$-modules, where $\calc$ is an essentially small category. Then we specialize to the cases of interest, related to the category $\fin$ of finite sets. This serves to introduce notation and some basic constructions.
 Unless otherwise stated, $\kring$ is a unital commutative ring. 
 
\subsection{Basics}
\label{subsect:basics}

Denote by $\fin$ the category of finite sets (and all maps), $\finj$ the wide subcategory of finite sets and injections, $\fs$ that of surjections, and $\fb$ that of bijections. Thus there is a commutative diagram of inclusions of wide subcategories:
\[
\xymatrix{
\fb 
\ar@{^(->}[r]
\ar@{^(->}[d]
&
\finj
\ar@{^(->}[d]
\\
\fs 
\ar@{^(->}[r]
&
\fin.
}
\]
 These categories each have skeleton given by the objects $\mathbf{n}:= \{ 1, \ldots, n\}$, for $n\in\nat$ (note that $\mathbf{0} = \emptyset$). The disjoint union $\amalg$ of finite sets is the coproduct in $\fin$; it restricts to symmetric monoidal structures on $\finj$,  $\fs$, and $\fb$. 

For $\calc$ an essentially small category, $\f (\calc)$ denotes the category of functors from $\calc$ to $\kmod$, the category of $\kring$-modules. It is equivalent to the category of $\kring$-linear functors from $\kring \calc$ (the $\kring$-linearization of $\calc$) to $\kmod$. Such functors will be referred to as $\kring \calc$-modules.

The category $\f(\calc)$ is abelian: a complex in $\f(\calc)$ is exact if and only if, for each $X \in\ob \calc$, the complex in $\kmod$ obtained by evaluating on $X$ is exact. 
This category is equipped with the tensor product $\otimes$ defined value-wise: for functors $F$, $G$ in $\f (\calc)$,  $(F \otimes G) (X) := F(X) \otimes G(X)$ (forming $\otimes := \otimes_\kring$ in $\kmod$).

\begin{example}
\label{exam:fb-modules_conv}
The category of $\kring \fb$-modules is equivalent to the category of symmetric sequences: objects are sequences $(M(n) \ | \ n \in \nat)$ of representations of the symmetric groups (i.e., for each $n$, $M(n)$ is a $\kring \sym_n$-module, where $\sym_n$ is the symmetric group on $n$ letters); morphisms are sequences of equivariant maps. Given a $\kring \sym_n$-representation $M$ for some $ n \in \nat$, this can be considered as a $\kring \fb$-module supported on sets of cardinal $n$.

In addition to the pointwise tensor product, the category of $\kring \fb$-modules is equipped with the Day convolution product $\odot$, which yields a symmetric monoidal structure. This is defined explicitly by $(M \odot N)(S) := \bigoplus_{S_1 \amalg S_2= S} M(S_1) \otimes N(S_2)$, where the sum runs over ordered decompositions of $S$ into two subsets. 

The category $\fb$ is a groupoid, thus the inverse gives an isomorphism of categories $\fb \cong \fb\op$ and hence of their linearizations. This implies that the category of $\kring \fb\op$-modules is isomorphic to that of $\kring \fb$-modules. The convolution product $\odot$ has a counterpart for $\kring \fb\op$-modules; where necessary, these will be distinguished by writing $\odot_\fb$ and $\odot_{\fb\op}$ respectively.
\end{example}

\begin{example}
\label{exam:otimes_fb}
The tensor product $\otimes_{\sym_n}$ for representations of the symmetric group $\sym_n$ has a `global' form
$
\otimes _\fb : \f(\fb\op) \times \f (\fb) \rightarrow \kmod$.
 It is given  for a $\kring \fb\op$-module $X$ and a $\kring \fb$-module $M$ by:
\[
X \otimes _\fb M 
:=
\bigoplus_{n\in \nat}
X (\n) \otimes_{\sym_n} M(\mathbf{n}).
\]
\end{example}

\begin{example}
\label{exam:kfin}
Forgetting structure, one can consider $\kring \fin  :  (S, T) \mapsto \kring \fin (S,T)$ as a $\kring \fb$-bimodule (i.e., a $\kring (\fb\op \times \fb)$-module) and likewise for $\kring \finj $ and $\kring \fs$. Since every map between finite sets $f: S \rightarrow T$ factors canonically as $S \twoheadrightarrow \mathrm{image}(f) \subset T$, 
one has the isomorphism of $\kring \fb$-bimodules:
\[
\kring \fin \cong \kring \finj \otimes_\fb \kring \fs,
\]
where the variance dictates how $\otimes_\fb$ is formed.

More precisely, this can be interpreted as an isomorphism of $\kring \finj \otimes \kring \fs\op$-modules, where the left hand side is given the module structure by restriction of the canonical $\kring \fin$-bimodule structure and the right hand side has the obvious structure. 
\end{example}

\begin{notation}
\label{nota:triv}
Write $\triv$ for the constant $\kring \fb$-module with value $\kring$, so that $\triv(\n)$ identifies as the trivial representation $\triv_n$ of $\kring \sym_n$. (The same notation will be used when working with $\kring \fb\op$-modules.)
\end{notation}

\begin{proposition}
\label{prop:identify_induction}
The composite of the functor $\kring \finj \otimes_\fb - : \f (\fb) \rightarrow \f (\finj)$ with the restriction $\f (\finj) \rightarrow \f (\fb)$ induced by the inclusion $\fb \subset \finj$ 
 is naturally equivalent to the functor 
$\triv \odot - : \f (\fb) \rightarrow \f (\fb)$.
\end{proposition}

\begin{proof}
We require to show that, for each $n \in \nat$, the functors $(\triv \odot - )(\n)$ and $\kring \finj (-,\n) \otimes_\fb -$ are naturally equivalent (taking values in $\kring \sym_n$-modules).

Now, for $s \in \nat$ (that we may assume to be at most $n$), $\kring \finj (\mathbf{s}, \mathbf{n})$ is isomorphic to the $\kring (\sym_n \times \sym_s\op)$-module $\kring \sym_n / \sym_{n-s}$, with $\sym_s$ acting on the right via the residual right regular action (use the inclusion of Young subgroup $\sym_{n-s} \times \sym_s \subset \sym_n$) and the left $\sym_n$-action is the usual one. Hence one sees that  $\kring \finj (\mathbf{s}, \mathbf{n}) \otimes_{\sym_s} M(s) $ is naturally isomorphic to 
$\triv_{n-s} \odot M(s)$. Summing over $s$ gives the result.   
\end{proof}

\begin{example}
\label{exam:kring_fin_odot}
By Example \ref{exam:kfin},  the underlying $\kring \fb$-bimodule of $\kring \fin$ identifies as 
$ 
\triv \odot_\fb \kring \fs$. 
\end{example}

\subsection{Projectives}

The category  $\f (\calc)$ has enough projectives. In particular, for an object $X$ of $\calc$, one has the standard projective $P^\calc_X$ in $\f (\calc)$ given by $\kring \hom_\calc (X, -) = \kring \calc( X, -)$; by Yoneda's lemma, this corepresents evaluation on $X$.

\begin{example}
\label{exam:standard_projective_fin}
For $n \in \nat$, one has the projective $\pfin_\n = \kring \fin (\n, -)$. Since $\n$ is isomorphic to the $n$-fold coproduct of copies of  $\mathbf{1}$, one has the isomorphism
\[
\pfin_\n \cong (\pfin_{\mathbf{1}})^{\otimes n}
\] 
that is $\kring \sym_n\op$-equivariant, where the symmetric group acts via automorphisms of $\n \cong \mathbf{1}^{\amalg n}$ on the left hand side and by place permutations of the tensor product on the right.

The functor $\pfin_{\mathbf{1}}$ identifies as $X \mapsto \kring [X]$, where $\kring [X]$ is the $\kring$-module freely generated by $X$; this comes with the canonical basis $\{ [x] \ | \ x \in X \}$. Hence $\pfin_\n$ identifies as the functor $X \mapsto \kring [X] ^{\otimes n}$.

Morphisms between standard projectives are understood by Yoneda's lemma, which gives:
\[
\hom_{\kring\fin} (\pfin_\n, \pfin_{\mathbf{t}}) \cong \kring \fin (\mathbf{t}, \n).
\]
This corresponds to the fully faithful embedding $
\kring \fin \op \hookrightarrow \f (\fin)$ sending $\n$ to $\pfin_\n$. 
\end{example}

Part of the relationship between $\kring \fin$ and $\kring \finj$ exhibited in Example \ref{exam:kfin} is made more precise in the following:

\begin{proposition}
\label{prop:pfinj_fin-module}
For $n \in \nat$, the canonical $\kring\finj$-module structure of $P^\finj_\n = \kring \finj (\n, -) $ extends to a $\kring\fin$-module structure. This is the unique structure such that the surjection 
\[
\pfin_\n = \kring \fin(\n, -) 
\twoheadrightarrow 
P^\finj_\n = \kring \finj (\n, -)
\]
induced by 
sending non-injective maps to zero and injective maps to themselves is a morphism of $\kring\fin$-modules. 
\end{proposition}

\begin{proof}
For each $t \in \nat$, one has $\fin^{\mathrm{non-inj} } (\n, \mathbf{t}) \subset \fin (\n , \mathbf{t})$, the set of non-injective maps. Since the property of being non-injective is preserved under post-composition, on passing to the $\kring$-linearization, one obtains a sub $\kring\fin$-module $\kring \fin^{\mathrm{non-inj} } (\n, -)$ of $\kring \fin (\n, -)$. The quotient by this submodule is thus a $\kring\fin$-module. By construction, on restricting to $\finj \subset \fin$, this quotient is isomorphic to $\kring \finj (\n, -)$. 
\end{proof}

\begin{remark}
\label{rem:left_aug}
This Proposition exhibits a left augmentation of the `unit' $\kring \finj \hookrightarrow \kring \fin$ (corresponding to the canonical inclusion), in the terminology of Positselski \cite{MR4398644}. Left (respectively right) augmentations are exploited in \cite{P_rel_kos}, using results of Positselski.  
\end{remark}

Proposition \ref{prop:pfinj_fin-module} provides a presentation of $P^\finj_\n$ considered as a $\kring\fin$-module:

\begin{corollary}
\label{cor:present_pfinj}
For $n\in \nat$, there is a presentation in $\f(\fin)$: 
\[
\bigoplus_{\substack{f \in \hom_{\kring \fin}  (\pfin_\mathbf{s} , \pfin_\n) \\
s < n}}
\pfin_{\mathbf{s}}
\rightarrow 
\pfin_\n
\rightarrow 
\kring \finj (\n, -) 
\rightarrow 
0,
\]
where the factor $\pfin_{\mathbf{s}}$ indexed by $f$ maps to $\pfin_\n$ by $f$.
\end{corollary}

\begin{proof}
An element $f \in \hom_{\kring\fin} (\pfin_\mathbf{s} , \pfin_\n) \cong \kring \fin (\n , \mathbf{s}) $ is a $\kring$-linear combination of maps in $\fin (\n , \mathbf{s})$. If $s<n$, then every map in $\fin (\n , \mathbf{s})$ is non-injective. In particular, the corresponding map $\pfin_\mathbf{s} \rightarrow \pfin_\n$ maps to the kernel of the surjection $\pfin_\n
\rightarrow 
\kring \finj (\n, -)$. Moreover, from the construction (also compare Example \ref{exam:kfin}), it follows that the image of the left hand map is $\kring \fin^{\mathrm{non-inj}}(\n, -)$, using the notation of the proof of Proposition \ref{prop:pfinj_fin-module}.
\end{proof}

There is a categorically dual version of Proposition \ref{prop:pfinj_fin-module} for surjective maps: 

\begin{proposition}
\label{prop:pfsop_finop-module}
For $n \in \nat$, the canonical $\kring\fs\op$-module structure of $P^{\fs\op}_{\mathbf{n}} = \kring \fs (-, \n)$ extends to a $\kring\fin\op$-module structure. This is the unique structure such that the surjection 
\[
P^{\fin\op}_\n = \kring \fin (-, \n) 
\twoheadrightarrow 
\kring \fs (-, \n) 
\]
induced by sending non-surjective maps to zero is a morphism of $\kring\fin\op$-modules.
\end{proposition}

\begin{proof}
The proof is categorically dual to that of Proposition \ref{prop:pfinj_fin-module}, using the fact that the property of being non-surjective is preserved under precomposition.
\end{proof}

\begin{remark}
\label{rem:right_aug}
Analogously to Remark \ref{rem:left_aug}, this Proposition exhibits a right augmentation of the unit $\kring \fs \hookrightarrow \kring \fin$ (corresponding to the canonical inclusion). 
\end{remark}

\subsection{A norm-like map}
\label{subsection:norm_map}

To simplify the exposition, in this subsection we suppose  that $\kring$ is a field.  Vector space duality induces an exact functor $D : \f (\calc)\op \rightarrow \f (\calc\op)$ where, for a $\kring\calc$-module $F$, the dual $DF$ is defined by $DF (X):= \hom_{\kmod} (F(X), \kring)$. 

Fix a finite set $X$. One has the inclusion of unital monoids $\aut (X) \hookrightarrow \fin (X,X)$. On passing to the $\kring$-linearizations, one also has the surjection of $\kring$-algebras 
\[
\kring \fin (X,X) \twoheadrightarrow \kring \aut (X)
\]
induced by sending non-bijective maps to zero. Clearly this has a section given by the canonical inclusion $\kring \aut (X) \hookrightarrow  \kring \fin (X,X)$. We can thus form the following composite, in which the first map is  induced by composition in $\fin$:
\begin{eqnarray}
\label{eqn:compose_to_aut}
\kring \fin (-,X) \otimes_\fin \kring \fin (X, -) 
\rightarrow 
\kring \fin (X,X) 
\rightarrow 
\kring 
\aut (X).
\end{eqnarray}
This composite factors using the projections of Propositions \ref{prop:pfinj_fin-module} and \ref{prop:pfsop_finop-module} to give 
\[
\kring \fs (-, X) \otimes_\fin \kring \finj (X, -)\rightarrow \kring \aut (X)
\]
and this is a morphism of $\kring \aut (X)$-bimodules.
 By adjunction, this corresponds to a morphism
\begin{eqnarray}
\label{eqn:adjoint_aut}
\kring \finj (X, - ) \rightarrow \hom_{\kring \aut(X) } (\kring \fs (-, X), \kring \aut (X)),
\end{eqnarray}
 of $\kring \aut (X)\op\otimes \kring \fin$-modules (for the action of $\kring \aut (X)\op$ on the codomain arising from the right regular action on $\kring \aut (X)$).

Using that coinduction is equivalent to induction for finite groups, one has:

\begin{lemma}
\label{lem:dual-bimodule}
There is an isomorphism of $\kring \aut (X)\op\otimes \kring \fin$-modules
\[
\hom_{\kring \aut(X) } (\kring \fs (-, X), \kring \aut (X))
\stackrel{\cong}{\rightarrow}  
D \kring \fs (-,X) = \hom_\kring (\kring \fs (-,X), \kring) 
\]
induced by the $\kring$-linear surjection $\kring \aut (X) \twoheadrightarrow \kring$ that sends a generator $[\alpha]$ ($\alpha \in \aut (X)$) to zero unless $\alpha = \id_X$, when the image is $1$.
\end{lemma}

Hence (\ref{eqn:adjoint_aut}) can be written as 
\begin{eqnarray}
\label{eqn:adjoint_aut_D}
\kring \finj (X, -) \rightarrow D \kring \fs (-, X). 
\end{eqnarray}

\begin{proposition}
\label{prop:nat_finj_fsop}
For a finite set $X$,  
the natural transformation (\ref{eqn:adjoint_aut_D}) is a morphism of $\kring \aut (X)\op\otimes \kring \fin$-modules.
 Evaluated on $X$, this gives an isomorphism
\[
\kring \finj (X,X)
\cong \kring \aut (X) 
 \stackrel{\cong}{\rightarrow} D \kring \fs (X,X) \cong D \kring \aut (X).
\] 
\end{proposition}

\begin{proof}
The first statement is an immediate consequence of the construction.

For the final statement, using $\fs (X,X)$ as the canonical basis of $\kring \fs (X,X)$, the dual basis provides the isomorphism of $\kring$-vector spaces $D \kring \fs (X,X) \cong \kring \fs(X,X)$, so that the map identifies as a $\kring$-linear map $
 \kring \aut (X) 
\rightarrow   \kring \aut (X)$.
 It is straightforward to check that this is given by $[\alpha] \mapsto [\alpha^{-1}]$, where $\alpha \in \aut (X)$. In particular, this is an isomorphism, as required.
\end{proof}

\begin{remark}
The construction of (\ref{eqn:adjoint_aut_D}) can also be carried out by first constructing the natural transformation 
\begin{eqnarray}
\label{eqn:norm_fin}
\kring \fin (X, -) 
\rightarrow 
D \kring \fin (-,X) 
\end{eqnarray}
adjoint to composition and then observing that it factors across   (\ref{eqn:adjoint_aut_D}) via 
the projections of Propositions \ref{prop:pfinj_fin-module} and \ref{prop:pfsop_finop-module}.
\end{remark}

\section{Filtering the standard projectives}
\label{sect:proj}

In this section we  commence the analysis of the structure of the projectives of $\f (\fin)$, working over a unital commutative ring $\kring$. 
 We first introduce the subfunctor $\pbar\subset \pfin_\mathbf{1}$, whence the tensor products $\pbar^{\otimes t}$ and $\pfin_\mathbf{1} \otimes \pbar^{\otimes t}$ for $t \in \nat$. 
These are key players in the paper; in particular, Proposition \ref{prop:splitting_pfin_1} gives a simple direct proof of the fact that $\pfin_\mathbf{1} \otimes \pbar^{\otimes t}$ is projective. In Section \ref{sect:proj_bis} it is explained how this  yields a more efficient set of projective generators of $\f (\fin)$.

The projective $\pfin_{\mathbf{0}}$ identifies as the constant functor with value $\kring$. There is a non-split short exact sequence 
\begin{eqnarray}
\label{eqn:ses_k}
0
\rightarrow 
\kbar 
\rightarrow 
\kring 
\rightarrow 
\kring_\mathbf{0}
\rightarrow 
0
\end{eqnarray}
where $\kring_\mathbf{0}$ is the functor  supported on $\mathbf{0}$ with value $\kring$  and $\kbar \subset \kring$ is the subfunctor supported on non-empty finite subsets. 

\begin{remark}
\label{rem:ses_kring}
If $\kring$ is a field, then both $\kring_\mathbf{0}$ and $\kbar$ are simple $\kring\fin$-modules.
\end{remark}

Now consider $\pfin_\mathbf{1}$, the functor $X \mapsto \kring [X]$ ($\kring [X]$ may also be written  as $\kring X$). By Yoneda, $\hom_{\kring \fin} (\pfin_{\mathbf{1}}, \pfin_\mathbf{0}) \cong \kring$, with generator 
\begin{eqnarray}
\label{eqn:proj_kX_k}
 \kring [X] \rightarrow \kring \cong \kring [*]
\end{eqnarray}
induced by the  map $X \rightarrow *$ to the terminal object of $\fin$. The image of this  map is $\kbar$.

\begin{notation}
\label{nota:pbar}
Denote by $\pbar$ the kernel of $\pfin_{\mathbf{1}} \twoheadrightarrow \kbar$.
\end{notation}

By construction one has the short exact sequence in $\f (\fin)$:
\begin{eqnarray}
\label{eqn:ses_pbar}
0
\rightarrow 
\pbar
\rightarrow 
\pfin_{\mathbf{1}}
\rightarrow 
\kbar 
\rightarrow 
0.
\end{eqnarray}

From the definition, the following is clear:

\begin{lemma}
\label{lem:basis_pbar}
For $X$ a non-empty finite set and $x \in X$, $\pbar (X) \subset \kring [X]$ has a non-canonical basis given by 
 $\big\{ [y]- [x] \ | \ y \in X \backslash \{ x \} \big\}$.  
\end{lemma}

\begin{lemma}
\label{lem:non-split}
The short exact sequence  (\ref{eqn:ses_pbar}) does not split. In particular, $\kbar$ is not projective. 
\end{lemma}

\begin{proof}
Were the short exact sequence to split, there would be a non-trivial surjection $\pfin_{\mathbf{1}} \twoheadrightarrow \pbar$. However, $\hom_{\f (\fin)} (\pfin_{\mathbf{1}} , \pbar)=0$, by Yoneda, since $\pbar(\mathbf{1}) =0$ by Lemma \ref{lem:basis_pbar}.
\end{proof}

Using the tensor product in $\f(\fin)$, for each $n \in \nat$, one can form the $n$-fold tensor product $\pbar^{\otimes n}$, adopting the following convention for $n=0$:

\begin{conv}
\label{conv:pbar_tensor_0}
The functor $\pbar^{\otimes 0}$ is taken to be $\kbar$.
\end{conv}

Forming the tensor product with $\kbar$ gives an idempotent endofunctor of $\f (\fin)$ (idempotency follows from the obvious isomorphism $\kbar \otimes \kbar \cong \kbar$). For a $\kring\fin$-module $F$, the inclusion $\kbar \subset \kring$ induces 
$
F \otimes \kbar \subset F
$,  
with image identifying as the subfunctor of $F$ supported on non-empty finite sets, which is denoted $\overline{F}$. Clearly $F \otimes \kbar$ is isomorphic to $F$ if and only if $F(\mathbf{0}) =0$. In particular one has:

\begin{lemma}
\label{lem:pfin_kbar}
For $n \in \nat$, $\pfin_\n \otimes \kbar \cong \pfin_\n$ if and only if $n>0$.
\end{lemma}

Then, for $n$ a positive integer, forming the complex $(\pfin_\mathbf{1} \rightarrow \kbar)^{\otimes n}$ gives the following generalization of the short exact sequence (\ref{eqn:ses_pbar}):

\begin{proposition}
\label{prop:exact_seq_pbar_otimes}
For a positive integer $n$, there is an exact sequence 
\[
0
\rightarrow 
\pbar^{\otimes n}
\rightarrow 
\pfin_\n
\rightarrow 
\bigoplus_{i=1}^n
\pfin_{\n\backslash \{ i \}}
\rightarrow 
\bigoplus_{\substack{T \subset \n \\ |T|=2 }}
\pfin_{\n\bs T}
\rightarrow 
\ldots 
\rightarrow 
\kbar 
\rightarrow 
0
\]
where the interior terms of the complex, indexed by  $t \in \{0,\ldots , n-1\}$,   are of the form $\bigoplus_{\substack{T \subset \n \\ |T|=t }}
\pfin_{\n\bs T}$; the differential between these is a (signed) sum of the morphisms induced by the inclusions $\n \bs T \subset \n \bs T'$ for $T' \subset T \subset \n$. 
\end{proposition}

Similarly, using the tensor product filtration for $(\pbar \subset \pfin_\mathbf{1})^{\otimes n}$, one has:

\begin{proposition}
\label{prop:filter_pfin_n}
For $0<n \in \nat$, $\pfin_\n$ admits a natural filtration in $\kring \sym_n\op \otimes \kring \fin$-modules:
\[
0= \filt_{n+1} \pfin_\n \subset \filt_n \pfin_\n  \subset \filt_{n-1} \pfin_\n \subset \ldots \subset \filt_0 \pfin_\n = \pfin_\n
\]
such that, for each $0 \leq k \leq n$,
\[
\filt_k \pfin_\n / \filt_{k+1} \pfin_\n \cong \pbar^{\otimes k} \odot \triv_{n-k}.
\]
\end{proposition}

\begin{remark}
The case $n=0$ could also be treated, at the expense of complicating the statement. This simply yields the filtration of $\pfin_{\mathbf{0}}$ corresponding to the short exact sequence (\ref{eqn:ses_k}).
\end{remark}

Lemma \ref{lem:pfin_kbar} also  yields the following projectivity result:

\begin{proposition}
\label{prop:splitting_pfin_1}
If $Q$ is a direct summand of $\pfin_\n$ for some $n>0$,   then $Q \otimes \pbar$ is projective in $\f (\fin)$ and is a direct summand of $\pfin_{\mathbf{n+1}}$.

Hence, for any $t \in \nat$, the functor $\pfin_\mathbf{1} \otimes \pbar^{\otimes t}$ is projective.
\end{proposition}

\begin{proof}
Applying $Q \otimes - $ to the short exact sequence (\ref{eqn:ses_pbar}) yields the short exact sequence 
\[
0
\rightarrow 
Q \otimes \pbar
\rightarrow 
Q \otimes \pfin_{\mathbf{1}}
\rightarrow 
Q \otimes \kbar 
\rightarrow 
0.
\]
The hypothesis upon $Q$ implies that $Q(\mathbf{0})=0$ and hence that $Q \otimes \kbar$ is isomorphic to $ Q$. Hence the short exact sequence in the statement splits, since $Q$ is projective. The final statement follows by using the isomorphism $\pfin_{\mathbf{n+1}} \cong \pfin_\n \otimes \pfin_{\mathbf{1}}$, so that $Q \otimes \pfin_{\mathbf{1}}$ is a direct summand of $\pfin_{\mathbf{n+1}}$ and is projective.

The projectivity of $\pfin_\mathbf{1} \otimes \pbar^{\otimes t}$ then follows by an obvious induction upon $t$.
\end{proof}

\begin{remark}
An alternative proof of the projectivity of $\pfin_\mathbf{1} \otimes \pbar^{\otimes t}$ is given in Section \ref{sect:proj_bis}, using the identification given in Proposition \ref{prop:llproj_identify}. 
\end{remark}

\begin{remark}
\label{rem:ses_Pfin_tensor_Pbars}
By construction (see Corollary \ref{cor:ses_llproj}), for $t \in \nat$, there is a short exact sequence:
\begin{eqnarray}
\label{eqn:ses_Pfin_tensor_Pbars}
0
\rightarrow 
\pbar ^{\otimes t+1} 
\rightarrow 
\pfin_{\mathbf{1}} \otimes \pbar^{\otimes t}
\rightarrow 
\pbar ^{\otimes t}
\rightarrow 
0. 
\end{eqnarray}
This does not split in general (if $\kring$ has characteristic zero, this  follows from the results of Section \ref{sect:charzero}), whence $\pbar ^{\otimes t}$ is not projective.

These short exact sequences splice to give a projective resolution of $\kbar$:
\begin{eqnarray}
\label{eqn:proj_resolution_kbar_general}
\ldots \rightarrow 
\pfin_{\mathbf{1}} \otimes \pbar^{\otimes t+1}
\rightarrow 
\pfin_{\mathbf{1}} \otimes \pbar^{\otimes t}
\rightarrow 
\ldots 
\rightarrow 
\pfin_{\mathbf{1}} \otimes \pbar
\rightarrow 
\pfin_{\mathbf{1}},
\end{eqnarray}
where $t$ corresponds to the homological degree. Splicing with $0\rightarrow \kbar \rightarrow \kring = \pfin_{\mathbf{1}} \rightarrow \kring _\mathbf{0} \rightarrow 0$ gives a projective resolution of $\kring_\mathbf{0}$.

Truncating by omitting terms in homological degree less than $n$, for $n$ a positive integer,  yields a projective resolution of $\pbar^{\otimes n}$.
\end{remark}

For later use, we record 
the following, which  gives further information on the surjection of  Proposition \ref{prop:pfinj_fin-module}:

\begin{proposition}
\label{prop:refine_surject_to_pfinj}
For $0< n \in \nat$, the composite 
$
\pfin_\mathbf{1} \otimes \pbar^{\otimes n-1}
\hookrightarrow 
\pfin_\n 
\twoheadrightarrow 
\kring \finj (\mathbf{n}, -) 
$
is surjective. 
\end{proposition}

\begin{proof}
It suffices to show that, for $f : \n \hookrightarrow X$ an injection, the generator $[f] \in \kring \finj (\mathbf{n}, X)$ lies in the image. 
 Write $x_i := f(i)$, for $i \in \n$. Then one has the element: 
\[
[x_1] \otimes ([x_2 ] - [x_1]) \otimes \ldots \otimes ([x_n ] - [x_1])
\in 
\big( \pfin_\mathbf{1} \otimes \pbar^{\otimes n-1} \big ) (X). 
\]
One checks directly that this maps to $[f]$, as required.
\end{proof}


\section{Splitting the projectives and applications}
\label{sect:proj_bis}

The purpose of this section is to revisit the projectives of the form $\pfin_\mathbf{1}\otimes \pbar^{\otimes n}$ and establish some of their properties.
 For this, in Section \ref{subsect:split_mobius_inversion}, the projective direct summands $\llproj^{X,x}_\emptyset$ of $\pfin_X$ are introduced, where $x \in X$. Their significance is established by Corollary \ref{cor:llproj_projective_generators} which shows how they yield a set of projective generators of $\f (\fin)$. A counterpart is established for $\f (\fin\op)$.

Section \ref{subsect:identify_llproj_llop} makes the structure of these projectives more explicit, in particular establishing the relation with the projectives introduced in Section \ref{sect:proj}. This analysis is essential input to the proof of  Theorem \ref{thm:morphisms_llproj}  in Section \ref{subsect:morphisms_llproj}. 

These results are then used in Section \ref{subsect:calculate_hom_pbar_otimes_pfin} to study morphisms from $\pbar^{\otimes n}$ to the standard projectives $\pfin_\mathbf{t}$, for $t \in \nat$, together with its applications. The main result of this subsection is Theorem \ref{thm:morphism_pbar_otimes_n_Pfin}.   

Finally, in Section \ref{subsect:finiteness_properties} these results are applied to establish some basic finiteness properties, also giving some initial information on the simple 
 $\kring \fin$-modules. 
 
Throughout this section, unless indicated otherwise, $\kring$ is a unital, commutative ring.

\subsection{Splitting $\pfin_X$ using Möbius inversion}
\label{subsect:split_mobius_inversion}

In this section a splitting of $\pfin_X$ is given by using the general, conceptual framework of M\"obius inversion. 

\begin{remark*}This approach (together with indications of proofs) was suggested by the referee, together with ingredients feeding into Theorem \ref{thm:morphisms_llproj}, which was proposed by the referee.
\end{remark*}

We work with finite pointed sets of the form $(X,x)$, so that $X$ is a finite set and $x \in X$. (This will allow the naturality of the constructions to be analysed should that be required.)

\begin{definition}
\label{defn:i,s,f,e}
Suppose that $(X, x)$ is a pointed finite set  and  $T,Z \subseteq X \bse{x}$ are subsets. 
\begin{enumerate}
\item 
Let $i^X_T \in \fin (X \bs T, X) $ be the inclusion $X \bs T \subset X$ and $s^{X,x}_T \in \fin (X, X \bs T) $ be the surjection given by 
\[
s^{X,x}_T \colon 
y \mapsto \left\{ 
\begin{array}{ll}
x & y \in T \\
y & \mbox{otherwise.}
\end{array}
\right.
\]
\item 
Define  $f^{X,x}_T \in \fin (X,X)$ by $f^{X,x}_T := i^X_T \circ s^{X, x}_T$.
\item 
Define $e^{X,x}_Z \in \kring \fin (X,X)$ by 
\[
e^{X,x}_Z:= \sum_{Z \subseteq Y \subseteq X \bse{x} } (-1) ^{|Y\bs Z|} [f^{X,x}_Y].
\]
\end{enumerate}
\end{definition}

The following records some basic properties of these morphisms for later use.

\begin{lemma}
\label{lem:basic_i,s,f,e}
For $(X, x)$ a finite pointed set and subsets $T, T' \subseteq X \bse{x}$, 
\begin{enumerate}
\item 
\label{item:s_circ_i}
$s^{X,x}_T \circ i^X_T = \id_{X \bs T}$ in $\fin (X\bs T , X \bs T)$; 
\item 
\label{item:f_circ_f}
$f^{X,x}_T \circ f^{X,x}_{T'} = f^{X,x}_{T\cup T'} = f^{X,x}_{T'} \circ f^{X,x}_T$ in $\fin (X,X)$;
\item 
\label{item:f_circ_e}
$f^{X,x}_T \circ e^{X,x}_T = e^{X,x}_T$  in $\kring \fin (X , X)$; 
\item 
\label{item:i_circ_e}
$i^X_T \circ e^{X \bs T, x}_\emptyset = e^{X,x}_T \circ i^{X,x}_T$ in $\kring \fin (X \bs T, X)$.
\end{enumerate}
\end{lemma}

\begin{proof}
Point (\ref{item:s_circ_i}) is clear and (\ref{item:f_circ_f}) is a straightforward verification. 
The property (\ref{item:f_circ_e}) is established readily by using (\ref{item:f_circ_f}) and property (\ref{item:i_circ_e}) is again a straightforward verification.
\end{proof}

The significance of the $e^{X,x}_T$ is established by the following, which is an immediate consequence of the Möbius inversion result \cite[Theorem A.1]{MR3466554}.

\begin{proposition}
\label{prop:mobius_inversion}
For $(X,x)$ a finite pointed set,  $\{ e^{X,x}_T \mid \emptyset \subseteq T \subseteq X \bse{x}\}$  is a set of orthogonal idempotents that sum to $[\id_X] \in \kring \fin (X,X)$.
\end{proposition}

\begin{remark}
A direct proof of the idempotency of $e^{X,x}_\emptyset$  was given  in a previous version of this paper, where the idempotent was denoted $\pi_\n$, for $n= |X \bse{x}|$. 
\end{remark}

\begin{definition}
\label{defn:llproj}
For $(X,x)$ a finite pointed set and $T \subseteq X \bse{x}$, define $\llproj^{X,x}_T$ by
\[
\llproj^{X,x}_T := \kring \fin (X, \bullet) e^{X,x}_T  \in \ob \f (\fin).
\]
\end{definition}

By Proposition \ref{prop:mobius_inversion}, there is a direct sum decomposition as a finite direct sum of projectives in $\f (\fin)$:
\begin{eqnarray}
\label{eqn:decompose_Pfin_X}
\pfin_X 
\cong 
\bigoplus_{\emptyset \subseteq Y \subseteq X \bse {x}} \llproj ^{X,x}_Y.
\end{eqnarray}

One can reduce to considering the projectives of the form $\llproj^{X\bs Y, x}_\emptyset$ using the following morphisms.

\begin{definition}
\label{defn:a_b}
For $(X,x)$ a pointed finite set and $T \subseteq X \bse{x}$, let 
\begin{enumerate}
\item 
$a^{X,x}_T \in \kring \fin (X \bs T, X)$ be the composite $e^{X,x}_T \circ i^X_T \circ e^{X \bs T, x}_\emptyset$; 
\item 
$b^{X,x}_T \in \kring \fin (X, X \bs T)$ be the composite $e^{X\bs T,x}_\emptyset \circ s^{X,x}_T \circ e^{X , x}_T$. 
\end{enumerate}
\end{definition}

These morphisms relate the idempotents $e^{X,x}_T$ and $e^{X \bs T, x} _\emptyset$ as follows:

\begin{proposition}
\label{prop:a_b}
For $(X,x)$ a pointed finite set and $T \subseteq X \bse{x}$, there are equalities: 
\begin{eqnarray*}
b^{X,x}_T \circ a ^{X,x}_T& = &e^{X \bs T, x}_\emptyset \mbox{ in } \kring \fin (X\bs T, X \bs T); 
\\
a^{X,x}_T \circ b ^{X,x}_T & = & e^{X , x}_T  \mbox{ in } \kring \fin (X, X ).
\end{eqnarray*}

These induce an isomorphism $\llproj^{X,x}_T \cong \llproj^{X\bs T,x} _\emptyset$.
\end{proposition}

\begin{proof}
Using Lemma \ref{lem:basic_i,s,f,e}, 
we have equalities 
\begin{eqnarray*}
b^{X,x}_T \circ a ^{X,x}_T &=& 
e^{X\bs T,x}_\emptyset \circ s^{X,x}_T \circ e^{X , x}_T \circ i^X_T \circ e^{X \bs T, x}_\emptyset\\
&=& 
e^{X\bs T,x}_\emptyset \circ s^{X,x}_T\circ i^X_T  \circ e^{X\bs T , x}_\emptyset  \circ e^{X \bs T, x}_\emptyset
\\
&=&
e^{X\bs T,x}_\emptyset.
\end{eqnarray*}

Likewise, we have equalities
\begin{eqnarray*}
a^{X,x}_T \circ b^{X,x}_T 
&=&
e^{X,x}_T \circ i^X_T \circ e^{X \bs T, x}_\emptyset  \circ s^{X,x}_T \circ e^{X,x}_T
\\
&=&
e^{X,x}_T  \circ e^{X, x}_T \circ i^X_T \circ s^{X,x}_T \circ e^{X,x}_T\\
&=&
e^{X,x}_T  \circ e^{X, x}_T \circ f^{X,x}_T \circ e^{X,x}_T.
\end{eqnarray*}
Since $f^{X,x}_T\circ e^{X,x}_T = e^{X,x}_T$, by Lemma \ref{lem:basic_i,s,f,e}, the second equality follows. 

Thus these induce isomorphisms
\[
\xymatrix{
\llproj^{X\bs T,x}_\emptyset 
\ar@<.75ex>[rr]^{a^{X,x}_T} 
\ar@{}[rr]|\cong 
&&
\llproj^{X,x}_T .
\ar@<.75ex>[ll]^{b^{X,x}_T}
}
\]
\end{proof}

\begin{corollary}
\label{cor:llproj_projective_generators}
For $(X,x)$ a finite pointed set, there is a direct sum decomposition:
\[
\pfin_X 
\cong 
\bigoplus_{\emptyset \subseteq Y \subseteq X \bse {x}} \llproj ^{X\bs Y,x}_\emptyset.
\]

Hence the set of objects $\{ \pfin_{\mathbf{0}} = \kring \} \amalg \{ \llproj^{\n, 1} _\emptyset \mid 0< n \in \nat \}$ is a set of projective generators of $\f (\fin)$.
\end{corollary}

\begin{proof}
The first statement follows from (\ref{eqn:decompose_Pfin_X}) by using Proposition \ref{prop:a_b}. The second statement follows by using the fact that  $\{ \pfin_\n \mid n \in \nat \}$ is a set of projective generators of $\f (\fin)$.
\end{proof}

\begin{remark}
In terms of the identification provided by Proposition \ref{prop:llproj_identify} below, this result was proved in an earlier version of this paper by a direct, ad hoc proof, using only the results of Section \ref{sect:proj}. The current proof makes clear how the result relates to M\"obius inversion, thus making the result more conceptual.
\end{remark}

Corollary \ref{cor:llproj_projective_generators} has a counterpart for the decomposition of the standard projectives in $\f (\fin\op)$ using the following projectives:

\begin{definition}
\label{defn:llop}
For $(X,x)$ a finite pointed set and $T \subseteq X \bse{x}$, define $\llop^{X,x}_T$ by 
\[
\llop^{X,x} _T := e^{X,x}_T \kring \fin (\bullet, X) \in \ob \f(\fin\op).
\]
\end{definition}

In the following, $\kring_\mathbf{0}$ is the $\kring\fin\op$-module supported on $\emptyset$ with value $\kring$.

\begin{corollary}
\label{cor:decompose_P^finop_X}
For $(X,x)$ a finite pointed set, there is a direct sum decomposition of projective objects in $\f (\fin\op)$:
\[
P^{\fin\op}_X 
\cong 
\bigoplus_{\emptyset \subseteq Y \subseteq X \bse {x}} \llop^{X \bs Y, x}_\emptyset.
\]

Hence, the set of objects $\{ P^{\fin\op}_\emptyset = \kring _\mathbf{0} \} \amalg \{ \llop ^{\n,1}_\emptyset \mid 0< n \in \nat \}$ is a set of projective generators of $\f (\fin\op)$.
\end{corollary}

\begin{remark}
For $(X,x)$ a finite pointed set and $T \subseteq X \bse{x}$, the functors $\llproj^{X,x}_T$ and $\llop^{X,x}_T$ are related by using Yoneda's lemma, as follows. Consider $\kring \fin = \kring \fin (-, -) $ as a bifunctor, i.e., an object of $\f (\fin\op\times \fin)$, then there are isomorphisms:
\begin{eqnarray*}
\llop ^{X,x}_T &\cong & \hom_{\kring \fin} (\llproj^{X,x}_T, \kring \fin) 
\\
\llproj ^{X,x}_T &\cong & \hom_{\kring \fin\op} (\llop^{X,x}_T, \kring \fin). 
\end{eqnarray*}
\end{remark}

\subsection{Identifying the functors $\llproj^{(X,x)}_\emptyset$ and $\llop^{(X,x)}_\emptyset$}
\label{subsect:identify_llproj_llop}

Fix $(X,x)$ a finite pointed set,  so that $X$ is non-empty. Using tensor products indexed by finite sets, there is a canonical isomorphism $\pfin_X \cong (\pfin_\mathbf{1})^{\otimes X}$; using the distinguished $x \in X$, the latter can be rewritten as $\pfin_\mathbf{1} \otimes (\pfin_\mathbf{1})^{\otimes X\bse{x}}$. 

\begin{lemma}
\label{lem:rho_surjection}
For $(X,x)$ a finite pointed set, there is a surjection
\begin{eqnarray*}
\pfin_X \cong \pfin_\mathbf{1} \otimes (\pfin_\mathbf{1})^{\otimes X\bse{x}}
& \stackrel{\rho^{X,x}}{\twoheadrightarrow} & 
\pfin_\mathbf{1} \otimes \pbar^{\otimes X\bse{x}}
\\
\ [y] \otimes \bigotimes_{j \in X \bse{x}} [y_j]
&
\mapsto 
&
[y] \otimes \bigotimes_{j \in X \bse{x}} ([y_j]-[y])
\end{eqnarray*}
where $[y_j]-[y]$ is considered as a section of $\pbar$ via Lemma \ref{lem:basis_pbar}.

For a finite set $Y$ and $\phi \in \fin (X, Y)$, $[\phi]$ lies in the kernel of $\rho^{X,x}(Y)$ if and only if $\phi^{-1} \phi (x)\supsetneq \{x \}$.
\end{lemma}

\begin{proof}
The morphism $\rho^{X,x}$ can be constructed by using Yoneda's lemma. The fact that it is surjective follows from Lemma \ref{lem:basis_pbar}.

The second statement follows from a direct calculation.
\end{proof}

\begin{proposition}
\label{prop:llproj_identify}
For $(X,x)$ a finite pointed set, the surjection $\rho^{X,x}$ induces an isomorphism 
\[
\llproj^{X,x}_\emptyset 
\cong 
\pfin_\mathbf{1} \otimes \pbar^{\otimes X\bse{x}}.
\]

In particular, for any finite set $Y$, $\llproj^{X,x}_\emptyset (Y) \subset \pfin_X (Y)$ is a free, finite rank $\kring$-module of rank $|Y| (|Y|-1) ^{|X\bse{x}|}$ with basis the set of elements $[\phi]e^{X,x}_\emptyset$, where $\phi \in \fin (X, Y)$ runs over the maps such that $\phi^{-1}\phi (x) = \{x \}$.
\end{proposition}

\begin{proof}
The first statement follows from the observation that the composite 
\[
\pfin_X \cong \pfin_\mathbf{1} \otimes (\pfin_\mathbf{1})^{\otimes X\bse{x}}
 \stackrel{\rho^{X,x}}{\twoheadrightarrow}  
\pfin_\mathbf{1} \otimes \pbar^{\otimes X\bse{x}}
\hookrightarrow 
\pfin_\mathbf{1} \otimes (\pfin_\mathbf{1})^{\otimes X\bse{x}}
\cong 
\pfin_X 
\]
(considered as an element of $\hom_{\kring \fin} (\pfin_X, \pfin_X) \cong \kring \fin (X,X)$) identifies with $e^{X,x}_\emptyset$, by inspection using the first part of  Lemma \ref{lem:rho_surjection}. The second statement then follows from the second statement of that  Lemma. 
\end{proof}

\begin{remark}
The second part of Proposition \ref{prop:llproj_identify} allows one to give an explicit description of the action of morphisms of $\fin$ on $\llproj^{X,x}_\emptyset$. (Similarly, a  description of the structure of the $\kring \fin\op$-module  $\llop^{X,x}_\emptyset$ can be derived from Proposition \ref{prop:llop_structure} below.)
\end{remark}

\begin{corollary}
\label{cor:ses_llproj}
For $(X,x)$ a finite pointed set, there is a short exact sequence 
\begin{eqnarray}
\label{eqn:ses_llproj}
0
\rightarrow 
\pbar^{\otimes X}
\rightarrow 
\llproj^{X,x}_\emptyset
\rightarrow 
\pbar^{\otimes X\bse{x}}
\rightarrow 
0
\end{eqnarray}
where, by convention, if $X \bse{x} = \emptyset$, $\pbar^{\otimes X \bse{x}}$ is taken to be $\kbar$.
\end{corollary}

\begin{proof}
By construction, one has the short exact sequence 
\[
0
\rightarrow 
\pbar 
\rightarrow 
\pfin_\mathbf{1}
\rightarrow 
\kbar
\rightarrow 
0.
\]
Upon forming the tensor product with $\pbar^{\otimes X\bse{x}}$ one obtains the required short exact sequence, by using the isomorphism of Proposition \ref{prop:llproj_identify}.
\end{proof}

One can also analyse the projective $\llop ^{X,x}_\emptyset$, considered as a subfunctor of $P^{\fin\op}_X \cong \kring \fin(\bullet, X)$.

\begin{proposition}
\label{prop:llop_structure}
For $(X,x)$ a finite pointed set and $Y$ a finite set, $\llop^{X,x}_\emptyset (Y)$ is a free, finite rank $\kring$-submodule of $\kring \fin(Y, X)$ with basis $e^{X,x}_\emptyset [\psi]$, where $\psi \in \fin (Y, X)$ ranges over the set of maps such that $\psi (Y) \cup \{x \} =X$.
\end{proposition}

\begin{proof}
One first checks that, if $\psi (Y) \cup \{x \} \subsetneq X$, then $e^{X,x}_\emptyset [\psi]=0$. Indeed, by definition, one has 
\[
e^{X,x}_\emptyset [\psi]
= 
\sum_{T \subseteq X \bse{x}} (-1)^{|T|} [f^{X,x}_T \circ \psi].
\]
By the hypothesis, there exists $w \in (X\bse{x})  \backslash \psi(Y)$ and the indexing set can be partitioned according to whether $w \in T$ or not. One checks directly that, for $Z \subseteq X \bse{x,w}$, $f^{X,x}_Z \circ \psi = f^{X,x}_{Z \amalg \{w\}} \circ \psi$, so that their contributions sum to zero in the expression for $e^{X,x}_\emptyset[\psi]$.  The claim follows.

It remains to check that the elements $e^{X,x}_\emptyset [\psi]$, where $\psi \in \fin (Y, X)$ ranges over the set of maps such that $\psi (Y) \cup \{x \} =X$, are linearly independent. 
For such a $\psi$, consider $\psi_T := f^{X,x}_T \circ \psi \in \fin (Y, X)$, for $T \subseteq X \bse{x}$. By definition of $f^{X,x}_T$, if $T \neq \emptyset$, $\exists t \in T$ such that $t \not \in \mathrm{image}f^{X,x}_T$, with $t \neq x$, so that $\psi_T (Y) \cup \{x \} \subsetneq X$. Writing 
\[
e^{X,x}_\emptyset [\psi]
= [\psi] + \sum_{T\neq \emptyset} (-1)^{|T|}[\psi_T],
\]
and using that  $\fin (Y,X)$ is a basis for $\kring \fin (Y,X)$, the result follows easily by neglecting all terms $[\phi]$ such that $\phi(Y) \cup \{x \} \subsetneq X$.
\end{proof}

\begin{corollary}
\label{cor:ses_llop}
For $(X,x)$ a finite pointed set, there is a  short exact sequence in $\f (\fin\op)$:
\[
0
\rightarrow 
\kring \fs (-, X\bse {x}) 
\rightarrow 
\llop^{X,x}_\emptyset
\rightarrow 
\kring \fs (-, X) 
\rightarrow 
0
\]
where the terms $\kring \fs (-, Z)$ (for $Z \in \{ X, X \bse{x}\}$) are equipped with the $\kring \fin \op$-module structure provided by Proposition \ref{prop:pfsop_finop-module}.

In particular, for a finite set $Y$,  $\llop^{X,x}_\emptyset (Y) =0$ if $|X|\geq |Y|+2$.
\end{corollary}

\begin{proof}
One first observes that, evaluating on a finite set $Y$, the result holds at the level of the underlying $\kring$-modules. This follows from the observation that, for $\psi \in \fin (Y, X)$ such that $\psi (Y) \cup \{x \}= X$, there are two mutually exclusive possibilities: either $\psi$ is surjective or $\psi (Y) = X \bse {x}$.

It remains to check that this corresponds to a short exact sequence of functors, as stated. This follows by checking that the map $\kring \fs (Y, X\bse {x}) \rightarrow R^{X,x}_\emptyset = e^{X,x}_\emptyset \kring \fin (Y, X)$ given by $[\psi] \mapsto e^{X,x}_\emptyset [\psi]$ defines a natural inclusion. This is straightforward.
\end{proof}

\subsection{Morphisms between functors of the form $\llproj_\emptyset^{X,x}$}
\label{subsect:morphisms_llproj}

The main result of this section, Theorem \ref{thm:morphisms_llproj}, gives information on morphisms between projectives of the form $\llproj_\emptyset^{X,x}$. It was proposed by the referee as a more general and more conceptual result from which the author's original Corollary \ref{cor:morphisms_pbar} can be derived.

\begin{thm}
\label{thm:morphisms_llproj}
For finite pointed sets $(X,x)$ and $(Y,y)$, there are natural isomorphisms:
\[
\hom_{\kring \fin} (\llproj^{X,x}_\emptyset, \llproj^{Y,y}_\emptyset)
\cong 
\left\{
\begin{array}{ll}
0 & |X|\geq |Y|+2 \\
\kring \fb (Y, X\bse{x}) & |X|= |Y|+1.
\end{array}
\right.
\]
\end{thm}

\begin{proof}
By Yoneda, we have  the identification 
\[
\hom_{\kring \fin} (\llproj^{X,x}_\emptyset, \llproj^{Y,y}_\emptyset)
= 
e^{X,x}_\emptyset \kring \fin (Y,X) e^{Y,y}_\emptyset
\]
and the right hand side can be rewritten as $\llop^{X,x}_\emptyset (Y) e^{Y,y}_\emptyset$, a submodule of  $\llop^{X,x}_\emptyset (Y)$.

The vanishing for $|X|\geq |Y|+2$ follows from that given in Corollary \ref{cor:ses_llop}.

When $|X|= |Y|+1$, that result gives that $\llop^{X,x}_\emptyset (Y)$ is isomorphic to $\kring \fs (Y, X\bse{x}) \cong \kring \fb (Y, X \bse{x})$ (using the cardinality hypothesis). It remains to check that this induces the isomorphism in the statement. 

As above, $\llop^{X,x}_\emptyset (Y)$ identifies with $\hom_{\kring \fin} (\llproj^{X,x}_\emptyset, \pfin _Y)$. Now, by Corollary \ref{cor:llproj_projective_generators},   there is a direct sum decomposition
\[
\pfin_Y
\cong 
\bigoplus_{\emptyset \subseteq Z \subseteq Y \bse {y}} \llproj^{Y \bs Z, y}_\emptyset.
\]
Since $\hom_{\kring \fin} (\llproj^{X,x}_\emptyset, \llproj^{Y \bs Z, y}_\emptyset)=0$ if $Z \neq \emptyset$, by the previous vanishing result together with the cardinality hypothesis, it follows that the inclusion $\llproj^{Y,y}_\emptyset \subset \pfin_Y$ induces an isomorphism 
\[
\hom_{\kring \fin} (\llproj^{X,x}_\emptyset, \llproj^{Y, y}_\emptyset)
\cong 
\hom_{\kring \fin}(\llproj^{X,x}_\emptyset, \pfin _Y)
\]
(this is  independent of $y \in Y$). The result follows.
\end{proof}

\begin{example}
\label{exam:construct_llproj_surj_mono}
Let $(X,x)$ be a finite pointed set with $|X|>1$. The inclusion $X \bse{x} \subset X$ induces the surjection $\pfin_X \twoheadrightarrow \pfin_{X\bse{x}}$. Choosing some $y \in X \bse{x}$ (this is possible, by the hypothesis on $|X|$), the argument used in the proof of Theorem \ref{thm:morphisms_llproj} shows that there is a (canonical) factorization:
\[
\xymatrix{
\llproj^{X,x}_\emptyset 
\ar@{-->}[r]
\ar@{^(->}[d]
&
\llproj^{X\bse{x}, y} 
_\emptyset 
\ar@{^(->}[d]
\\
\pfin_X 
\ar@{->>}[r]
&
\pfin_{X \bse {x}}.
}
\]
Moreover, by inspection, the dotted map factors as 
\[
\llproj^{X,x} _\emptyset 
\twoheadrightarrow 
\pbar^{\otimes X\bse {x}} 
\hookrightarrow 
\llproj^{X\bse{x}, y}_\emptyset.
\] 	
This yields an explicit construction of the maps appearing in the short exact sequence (\ref{eqn:ses_llproj}) of Corollary \ref{cor:ses_llproj}.
\end{example}

\begin{corollary}
\label{cor:morphisms_pbar}
For $s, t \in \nat$, if $s>t$ then 
$
\hom_{\kring \fin} (\pbar^{\otimes s} , \pbar^{\otimes t}) =0
$.   
 For $s=t$,  the place permutation action on $\pbar ^{\otimes t}$ induces an isomorphism  of $\kring$-algebras:
$
\kring \sym_t
\cong 
\mathrm{End}_{\kring \fin} (\pbar ^{\otimes t})$.
\end{corollary}

\begin{proof}
Using the morphisms appearing in (\ref{eqn:ses_llproj}) (as reconstructed in Example \ref{exam:construct_llproj_surj_mono}), one has the morphisms 
$
\llproj^{\mathbf{s+1}, s+1}_\emptyset \twoheadrightarrow  \pbar^{\otimes s}$ and $\pbar^{\otimes t} \hookrightarrow  \llproj^{\mathbf{t}, t}_\emptyset$.
 These induce the inclusion
\[
\hom_{\kring\fin} (\pbar^{\otimes s} , \pbar^{\otimes t}) 
\hookrightarrow 
\hom_{\kring \fin} (\llproj^{\mathbf{s+1}, s+1}_\emptyset, \llproj^{\mathbf{t}, t}_\emptyset).
\]
This immediately implies the vanishing for $s>t$, using the vanishing statement of Theorem \ref{thm:morphisms_llproj}.

The case $s=t$ follows by verifying that the composite $\kring \sym_t \hookrightarrow 
\hom_{\kring\fin} (\pbar^{\otimes t} , \pbar^{\otimes t}) 
\hookrightarrow 
\hom_{\kring \fin} (\llproj^{\mathbf{t+1}, t+1}_\emptyset, \llproj^{\mathbf{t}, t}_\emptyset)$, where the first map is given by the place permutation action, is the isomorphism given by Theorem \ref{thm:morphisms_llproj}. This is straightforward.
\end{proof}

 \begin{remark}
 \ 
 \begin{enumerate}
 \item 
 The calculation in the case $s < t$ is much more delicate. This was first carried out in \cite{MR4518761}; a direct approach is given in Part \ref{part:calculate}
  working over a field of characteristic zero (see Theorem \ref{thm:endo_pbar_otimes}).
 \item 
 Theorem \ref{thm:morphisms_llproj} and 
Corollary \ref{cor:morphisms_pbar}  
 are our  substitutes for the {\em squishing} arguments employed by Wiltshire-Gordon in \cite{2014arXiv1406.0786W}. 
\end{enumerate} 
  \end{remark}

\subsection{Calculating $\hom_{\kring \fin} (\pbar^{\otimes n}, \pfin_\mathbf{t})$ and applications}
\label{subsect:calculate_hom_pbar_otimes_pfin}

For $n \in \nat$, we consider $\pbar^{\otimes n}$ (using our convention that, for $n=0$, this is $\kbar$);  the symmetric group $\sym_n$ acts on $\pbar^{\otimes n}$ by place permutations. 

 Putting together Corollary \ref{cor:ses_llproj} with the identification given in Example \ref{exam:construct_llproj_surj_mono}, we have a projective presentation of $\pbar^{\otimes n}$:
\begin{eqnarray}
\label{eqn:proj_presentation_pbar_otimes_n}
\llproj^{\mathbf{n+2}, n+2} _\emptyset
\rightarrow 
\llproj^{\mathbf{n+1}, n+1} _\emptyset
\rightarrow 
\pbar^{\otimes n}
\rightarrow 
0.
\end{eqnarray}
The first map is induced by the canonical inclusion $\mathbf{n+1} \subset \mathbf{n+2}$, as in Example \ref{exam:construct_llproj_surj_mono}.

Using this presentation, we prove the following:

\begin{thm}
\label{thm:morphism_pbar_otimes_n_Pfin}
For $n \in \nat$, there is an isomorphism of $\kring \sym_n \otimes \kring \fin\op$-modules
\[
\hom_{\kring \fin}(\pbar^{\otimes n}, \pfin_\bullet) 
\cong 
\kring \fs (\bullet, \n), 
\]
where the $\kring \fin\op$-module structure of the codomain is given by Proposition \ref{prop:pfsop_finop-module} and its $\kring \sym_n$-module structure by the action of $\sym_n$ on $\n$.
\end{thm}

\begin{proof}
Applying the functor $\hom_{\kring \fin}(-, \pfin_\bullet) $ to the presentation (\ref{eqn:proj_presentation_pbar_otimes_n}) gives the exact sequence 
\[
0
\rightarrow 
\hom_{\kring \fin}(\pbar^{\otimes n}, \pfin_\bullet) 
\rightarrow 
\llop^{\mathbf{n+1},n+1} _\emptyset 
\rightarrow 
\llop^{\mathbf{n+2}, n+2}_\emptyset.
\]
Here the structure of the two right hand terms is described by Proposition \ref{prop:llop_structure}, with complements in Corollary \ref{cor:ses_llop}. The  morphism between these is induced by the inclusion $\mathbf{n+1}\subset \mathbf{n+2}$. It follows that this factors as 
\begin{eqnarray}
\label{eqn:factorization_llop}
\llop^{\mathbf{n+1},n+1} _\emptyset 
\twoheadrightarrow
\kring \fs (\bullet, \mathbf{n+1}) 
\hookrightarrow
\llop^{\mathbf{n+2}, n+2}_\emptyset,
\end{eqnarray}
where the surjection and inclusion are those appearing in the short exact sequence of Corollary \ref{cor:ses_llop} for $(\mathbf{n+1}, n+1)$ and $(\mathbf{n+2}, n+2)$ respectively. 

Hence, by the short exact sequence of that Corollary, one obtains the required isomorphism of $\kring \fin\op$-modules. One checks directly that this is $\kring \sym_n$-equivariant, as required. 
\end{proof}

The argument shows more: 

\begin{corollary}
\label{cor:ses_llproj_llop}
For $n \in \nat$, applying the functor $\hom_{\kring \fin}(-, \pfin_\bullet) $ to the short exact sequence 
\[
0
\rightarrow 
\pbar^{\otimes n+1}
\rightarrow 
\llproj^{\mathbf{n+1}, n+1}_\emptyset
\rightarrow 
\pbar^{\otimes n}
\rightarrow 
0
\]
of Corollary \ref{cor:ses_llproj}
 yields the short exact sequence 
\[
0
\rightarrow 
\kring \fs (\bullet, \mathbf{n}) 
\rightarrow 
\llop _\emptyset^{\mathbf{n+1}, n+1}
\rightarrow 
\kring \fs (\bullet, \mathbf{n+1})
\rightarrow 
0
\]
of Corollary \ref{cor:ses_llop}.
\end{corollary}

\begin{proof}
Using Theorem \ref{thm:morphism_pbar_otimes_n_Pfin} and the left exactness of $\hom_{\kring \fin}(-, \pfin_\bullet) $, the only point to establish is that the inclusion $\pbar^{\otimes n+1}
\hookrightarrow 
\llproj^{\mathbf{n+1}, n+1}_\emptyset$ induces the surjection $\llop _\emptyset^{\mathbf{n+1}, n+1}
\twoheadrightarrow 
\kring \fs (\bullet, \mathbf{n+1})$. This follows from the factorization given in (\ref{eqn:factorization_llop}).
\end{proof}

\begin{corollary}
\label{cor:model_right_augmentation}
For $n \in \nat$, the map 
\begin{eqnarray*}
\hom_{\kring \fin} (\pfin_\n, \pfin_\bullet) 
\rightarrow 
\hom_{\kring \fin} (\pbar^{\otimes n}, \pfin_\bullet)
\end{eqnarray*}
induced by the inclusion  $\pbar^{\otimes n} \subset \pfin_\n$ identifies with the right augmentation
\[
\kring \fin (\bullet, \n) \twoheadrightarrow \kring \fs (\bullet, \n)
\]
of Proposition \ref{prop:pfsop_finop-module}.
\end{corollary}

\begin{proof}
The inclusion  $\pbar^{\otimes n} \subset \pfin_\n$ factors across $\llproj_\emptyset^{\n, i} \subset \pfin_\n$ for any $i \in \n$. By Corollary \ref{cor:ses_llproj_llop}, applying $\hom_{\kring \fin} (-, \pfin_\bullet)$ yields surjections 
\[
\kring \fin (\bullet, \n) \twoheadrightarrow \llop ^{\n, i} _\emptyset (\bullet) \twoheadrightarrow \kring \fs (\bullet, \n).
\]
It remains to check that this composite coincides with the right augmentation. By Yoneda, it suffices to check the image of $[\id_\n] \in \kring \fin (\n, \n)$; this is straightforward.
\end{proof}

These results also allow us to recover information on morphisms between tensor products of $\pbar$:

\begin{corollary}
\label{cor:hom_pbar_otimes}
For $0<a  \in\nat$, and $n \in \nat$, $\hom_{\kring \fin } (\pbar^{\otimes n}, \pbar^{\otimes a})$ is isomorphic as a $\kring \sym_n \otimes \kring \sym_a\op$-module to the kernel of the total restriction map 
\[
\kring \fs (\mathbf{a}, \n) 
\rightarrow 
\bigoplus_{i=1}^a
\kring \fs (\mathbf{a}\backslash \{ i \}, \n). 
\]
(Here, for a surjection $f : \mathbf{a} \twoheadrightarrow \n$, the $i$th component of the image of $[f]$ is $[f|_{ \mathbf{a}\backslash \{ i \}}]$, understood to be zero if $f|_{ \mathbf{a}\backslash \{ i \}}$ is not surjective.)
\end{corollary}

\begin{proof}
By Proposition \ref{prop:exact_seq_pbar_otimes}, there is an  exact sequence $$
0
\rightarrow 
\pbar^{\otimes a}
\rightarrow 
\pfin_\mathbf{a}
\rightarrow 
\bigoplus_{i=1}^a
\pfin_{\mathbf{a}\backslash \{ i \}}.
$$ 
The result follows by applying $\hom_{\kring \fin} (\pbar^{\otimes n} , -)$ and using Theorem \ref{thm:morphism_pbar_otimes_n_Pfin} to identify the terms and the maps. 
\end{proof}

\begin{remark}
Corollary \ref{cor:hom_pbar_otimes} was proved using different methods in \cite{MR4518761}. 
\end{remark}

 \subsection{Finiteness properties}
\label{subsect:finiteness_properties}

The above results imply finiteness properties for certain $\kring \fin$-modules. In preparation, we note the following basic property: 

\begin{lemma}
\label{lem:vanishing_fin_modules}
For a $\kring \fin$-module $F$ and positive integers $m \leq n$, if $F (\n)=0$ then $F(\m)=0$. 
\end{lemma}

\begin{proof}
The standard inclusion $\m \subset \n$ clearly admits a retract in $\fin (\n , \m)$, since $\m \neq \emptyset$ by hypothesis. Hence the $\kring$-module $F(\m)$ is a direct summand of $F (\n)$.
\end{proof}
 
\begin{remark}
The example of $\kring_\mathbf{0}$ shows that the hypothesis upon $m$ is necessary.
\end{remark}

We have the following general consequence of Corollary \ref{cor:llproj_projective_generators}.
 
\begin{proposition}
\label{prop:subfunctors_pbar_otimes_t}
For $0<t\in \nat$ and subfunctors $F \subset G \subset \pbar^{\otimes t}$, if $F(\mathbf{t+1}) \subset G(\mathbf{t+1})$ is an equality then $F =G$.
\end{proposition} 
 
\begin{proof}
We require to prove that the quotient $G/F$ is zero.  We have $\pbar^{\otimes t} (\mathbf{0})=0$, so that the hypothesis implies that  $G/F(\n)=0$ for each $n\leq t+1$, by Lemma \ref{lem:vanishing_fin_modules}.
 
Suppose that $G/F \neq 0$. By the above, we have $\hom_{\kring \fin} (\kring, G/F)=0$,  thus the set of projective generators given in Corollary \ref{cor:llproj_projective_generators} implies that there exists a pointed finite set $(X,x)$ and a non-zero morphism $\llproj^{X,x} _\emptyset \rightarrow G/F$. Since $\llproj^{X,x}_\emptyset$ is projective, this lifts to a non-zero morphism $\llproj^{X,x} _\emptyset \rightarrow G$; composing with $G \subset \pbar^{\otimes t} $ gives a non-trivial morphism $\llproj^{X,x} _\emptyset \rightarrow \pbar^{\otimes t}$. Since $\pbar^{\otimes t}$ is a subfunctor of $\llproj^{\mathbf{t}, t}_\emptyset$ (using that $t>0$), by Corollary \ref{cor:ses_llproj}, this yields a non-trivial morphism $\llproj^{X,x} _\emptyset \rightarrow \llproj^{\mathbf{t}, t}_\emptyset$. The vanishing result of Theorem \ref{thm:morphisms_llproj} therefore implies that $|X| \leq t+1$. 

Since $\llproj_\emptyset^{X,x}$ is a direct summand of $\pfin_X$, Yoneda's lemma gives that $\hom_{\kring \fin} (\llproj_\emptyset^{X,x}, G/F)$ is a direct summand of $G/F(X)$. However, the latter was shown to be  zero, which establishes a contradiction. Thus $G/F =0$, as required.
 \end{proof}
 
Recall that an object $A$ in an abelian category is simple (or irreducible) if its only subobjects are $0$ and $A$.  
 A similar argument to the above then shows the following:

\begin{proposition}
\label{prop:simples}
Suppose that $F\neq 0$ is a $\kring \fin$-module such that $F(\mathbf{0})=0$. Then there exists $t \in \nat$ and a non-trivial morphism $\pbar^{\otimes t} \rightarrow F$.
In particular, if $F$ is a simple functor, then $F$ is a quotient of $\pbar^{\otimes t}$. 

If $F(\mathbf{0})\neq 0$, there is a surjection 
$
F \twoheadrightarrow F (\mathbf{0})$ in $\f (\fin)$. 
Hence $F$ is a simple $\kring \fin$-module if and only if $F = F(\mathbf{0})$ and the latter is a simple $\kring$-module.
\end{proposition} 
 
 \begin{proof}
 For the first statement, as in the proof of Proposition \ref{prop:subfunctors_pbar_otimes_t}, since $F(\mathbf{0})=0$, there exists a finite pointed set $(X,x)$ and a non-zero morphism $\llproj^{X,x}_\emptyset \rightarrow F$. Using the short exact sequence of Corollary \ref{cor:ses_llproj}, either this restricts to a non-zero morphism $\pbar^{\otimes X}\rightarrow F$ or it factors across a non-zero morphism $\pbar^{\otimes X \bse{x}} \rightarrow F$. The conclusion in the case that $F$ is simple is immediate.

The second statement is straightforward.
 \end{proof}
 
Recall that a composition series of an object $A$ in an abelian category is a finite length filtration 
\[
0=A_0 \subset A_1 \subset \ldots \subset A_\ell =A
\]
with the property that each $A_{i+1}/A_i$ is a simple object. By the Jordan-H\"older theorem, the associated graded $\bigoplus_{i=0}^{\ell -1} A_{i+1}/ A_i$ is independent (up to isomorphism) of the choice of composition series.  The composition factors of $A$ are the simple functors $A_{i+1}/A_i$, which can occur with multiplicity greater than one.

\begin{corollary}
\label{cor:composition_series_pbar_otimes_t}
Suppose that $\kring$ is a field and that  $t$ is a positive integer. Then $\pbar^{\otimes t}$ admits a (finite length) composition series and, for each composition factor $S$ of $\pbar^{\otimes t}$:
\begin{enumerate}
\item 
$S (\mathbf{t+1})\neq 0$; 
\item 
$S$ is a quotient of $\pbar^{\otimes n}$ for some $n \leq t+1$.
\end{enumerate}
\end{corollary}

\begin{proof}
The $\kring$-vector space $\pbar^{\otimes t}(\mathbf{t+1})$ has finite dimension, so that this follows directly from the two previous Propositions and their proofs.
\end{proof}

Using Corollary \ref{cor:ses_llproj}, this implies:

\begin{corollary}
\label{cor:finite_length_indec_projectives}
For $\kring $ a field, the indecomposable projectives of $\f (\fin)$ admit (finite length) composition series. 
\end{corollary}

\begin{proof}
By Corollary \ref{cor:llproj_projective_generators}, an indecomposable projective of $\f (\fin)$ is either isomorphic to  $\pfin_\mathbf{0} = \kring$ or is a direct summand of some $\llproj^{\n,n}_\emptyset$, for $0<n \in \nat$. The functor $\kring$ has composition series exhibited by (\ref{eqn:ses_k}), using Remark \ref{rem:ses_kring}. In the second case, the result follows from Corollary \ref{cor:composition_series_pbar_otimes_t} in conjunction with the short exact sequence of Corollary \ref{cor:ses_llproj}.
\end{proof}

\part{The simple $\kring\fin$-modules in characteristic zero}
\label{part:zero}

The purpose of this part is to revisit the classification of the simple $\kring \fin$-modules and their projective covers,  working over a field $\kring$ of characteristic zero. In this case, the category of $\kring \sym_t$-modules is semi-simple, for all $ t\in \nat$. Isomorphism classes of simple $\kring \sym_t$-modules are indexed by the partitions $\lambda \vdash t$, with the corresponding simple module denoted $S_\lambda$. (Our convention is that $(1^t)$ indexes the sign representation $\sgn_t$.) 

\section{Projectives in characteristic zero}
\label{sect:charzero}

In this section, we study the indecomposable projectives of the category $\f (\fin)$. Apart from the projective $\pfin_\mathbf{0} \cong \kring$, these are indexed by partitions $\lambda \vdash n$ for $0<n$. The corresponding classification of the simple functors in $\f (\fin)$ is treated in Section \ref{sect:simples}.
 The case of the  partitions $(1^n)$ is exceptional, accounting for the subtleties of the theory.

Theorem \ref{thm:projectives_pbar} introduces the objects $S_\lambda (\pbar)$; these are projective for $\lambda$ not of the form $(1^n)$. For $\lambda= (1^n)$, $S_\lambda (\pbar)= \Lambda^n (\pbar)$ is not projective, where $\pfin$ is shorthand for $\pfin_{\mathbf{1}}$; its projective cover is identified as $\Lambda^{n+1}(\pfin)$ in Theorem \ref{thm:projective_cover_Lambda}, which also describes the structure of this projective. (The latter results are contained in the work of Wiltshire-Gordon \cite{2014arXiv1406.0786W}.)

\begin{remark}
Recall that Corollary \ref{cor:finite_length_indec_projectives} established that the indecomposable projective objects in $\f (\fin)$ have (finite length) composition series. 
\end{remark}

\subsection{A family of indecomposable projectives}

For $t$ a positive integer and $\lambda \vdash t$, consider the simple $\kring \sym_t$-module $S_\lambda$. 
 One has the associated Schur functor on $\kring$-vector spaces $V \mapsto S_\lambda (V):=  V^{\otimes t} \otimes _{\sym_t} S_\lambda$, where $\sym_t$ acts on $V^{\otimes t}$ by place permutations on the right. By post-composition,  this gives  
$S_\lambda (\pfin)$ and $S_\lambda (\pbar)$. When $\lambda = (1^t)$, these will be written as $\Lambda^t (\pfin)$ and $\Lambda^t (\pbar)$ respectively, since the Schur functor in this case identifies as the exterior power functor $\Lambda^t(-)$. 

The regular left and right actions make $\kring \sym_t$ into a $\kring \sym_t$-bimodule. This identifies as 
\begin{eqnarray}
\label{eqn:ksym_t_bimodule}
\kring \sym_t \cong \bigoplus_{\lambda \vdash t} S_\lambda \boxtimes S_\lambda ,
\end{eqnarray}
using the exterior tensor product of representations (where the second factor is considered as a $\kring \sym_t\op$-module); each $S_\lambda \boxtimes S_\lambda$ is a simple bimodule.
 Forgetting the left $\sym_t$-action, this gives the decomposition 
 $
\kring \sym_t \cong \bigoplus_{\lambda \vdash t}  S_\lambda^{\oplus \dim S_\lambda}
 $
as $\kring \sym_t\op $-modules.
 It follows that one has direct sum decompositions
\begin{eqnarray*}
\pfin_{\mathbf{t}} &\cong & \bigoplus_{\lambda \vdash t}   S_\lambda (\pfin) \boxtimes S_\lambda \cong 
\bigoplus_{\lambda \vdash t}  S_\lambda (\pfin)^{\oplus \dim S_\lambda}\\
(\pbar)^{\otimes t} &\cong & \bigoplus_{\lambda \vdash t} S_\lambda (\pbar ) \boxtimes S_\lambda
\cong 
\bigoplus_{\lambda \vdash t}  S_\lambda (\pbar)^{\oplus \dim S_\lambda}
.
\end{eqnarray*}
In each case,  the first expression gives the  $ \kring \fin \otimes \kring \sym_t\op $-module structure, whereas the second only records the underlying $\kring \fin$-module.

\begin{proposition}
\label{prop:decompose_tensor_pbar}
For $t \in \nat$, the decomposition 
$
\pbar^{\otimes t} 
\cong 
\bigoplus_{\lambda \vdash t}  S_\lambda (\pbar)^{\oplus \dim S_\lambda}
$ is the complete splitting of $\pbar^{\otimes t}$ into indecomposable objects.
 Moreover, for partitions $\lambda, \mu \vdash t$, 
\[
\hom _{\kring \fin} (S_\lambda (\pbar), S_\mu (\pbar)) \cong 
\left\{ 
\begin{array}{ll}
\kring & \lambda = \mu \\
0 & \lambda \neq \mu.
\end{array}
\right.
\]
\end{proposition}

\begin{proof}
By Corollary \ref{cor:morphisms_pbar}, the endomorphism ring of $\pbar^{\otimes t}$ in $\f (\fin)$ is $\kring \sym_t$. Let $e_\lambda$ and $e_\mu$ be idempotents in $\kring \sym_t$ so that there are isomorphisms of $\kring \sym_t$-modules $S_\lambda \cong \kring \sym_t e_\lambda$ and $S_\mu \cong \kring \sym_t e_\mu$. Thus there are isomorphisms $S_\lambda (\pbar) \cong \pbar^{\otimes t} e_\lambda$ and $S_\mu (\pbar ) \cong \pbar^{\otimes t} e_\mu$ so that one has the isomorphism of $\kring$-vector spaces 
\[
\hom _{\kring \fin} (S_\lambda (\pbar), S_\mu (\pbar)) \cong 
e_\lambda \kring \sym_t e_\mu.
\]
Using the identification of $\kring \sym_t$ as a bimodule given by (\ref{eqn:ksym_t_bimodule}), the result follows.
\end{proof}

It is immediate that $S_\lambda (\pfin)$ is projective in $\f (\fin)$, for each $\lambda \vdash t$, since it is a direct summand of $\pfin_\mathbf{t}$. 
 However, a stronger result holds for partitions not of the form $(1^t)$ (the case $\lambda = (1^t)$ is treated in the following subsection):

\begin{thm}
\label{thm:projectives_pbar}
Suppose that $\lambda \vdash t$ is not the partition $(1^t)$. Then $S_\lambda (\pbar)$ is an indecomposable projective in $\f (\fin)$. 
\end{thm}

\begin{proof}
The result is proved by induction upon $t$, starting from the initial case $t=2$, for which the only partition to consider is $\lambda = (2)$. In this case, $S_{(2)} (\pbar)$ is the second symmetric power $S^2 (\pbar)$ and one has the decomposition $\pbar^{\otimes 2} \cong \Lambda^2 (\pbar) \oplus S^2 (\pbar)$. Moreover, $S^2 (\pbar)$ is isomorphic to the invariants $(\pbar^{\otimes 2}) ^{\sym_2}$ for the place permutation action. 

To prove that $S^2 (\pbar)$ is projective, it suffices to define a suitable idempotent in $\mathrm{End}_{\kring \fin} (\pfin_\mathbf{2})$ inducing a projection onto $S^2 (\pbar)$. Now, $\mathrm{End}_{\kring \fin} (\pfin_\mathbf{2})$   identifies  as a $\kring$-vector space as $\pfin_\mathbf{2}(\mathbf{2}) $, by Yoneda. Using this isomorphism, one checks that $- \frac{1}{2} \big( ([1]-[2]) \otimes ([1 ] - [2]) \big)\in  \pfin_\mathbf{2}(\mathbf{2})$ is such an idempotent.

For the inductive step,  for a partition $\mu \vdash t+1$ such that $\mu \neq (1^{t+1})$, by Pieri's rule, there exists a partition $\lambda \vdash t$ such that $\lambda \neq (1^t)$ and $S_\mu$ is a direct summand of the induced module $S_\lambda\uparrow_{\sym_t}^{\sym_{t+1}}$. This implies that $S_\mu (\pbar)$ is a direct summand of 
$S_\lambda (\pbar ) \otimes \pbar$. By the inductive hypothesis, $S_\lambda (\pbar )$ is a projective and a direct summand of $\pfin_{\mathbf{t}}$, where $t>0$.  Proposition \ref{prop:splitting_pfin_1} thus applies to show that $S_\mu (\pbar)$ is projective and a direct summand of $\pfin_{\mathbf{t+1}}$. 

Finally, Proposition \ref{prop:decompose_tensor_pbar} implies that the endomorphism ring of $S_\lambda (\pbar) $ is $\kring$. This implies that $S_\lambda (\pbar)$ is indecomposable.
\end{proof}

\begin{remark}
\ 
\begin{enumerate}
\item 
An alternative proof of the Theorem is based on Lemma \ref{lem:lambda_pfin} below. In particular, this exhibits a surjection $\pfin \otimes \pbar \twoheadrightarrow \Lambda^2 (\pfin)$ between projective functors, using Proposition \ref{prop:splitting_pfin_1} for the projectivity of the domain. The kernel identifies with $S^2 (\pbar)$, which is thus projective, since it is a direct summand of the projective $\pfin \otimes \pbar$, by the splitting of the short exact sequence. 
 By induction, this argument generalizes using Lemma \ref{lem:lambda_pfin} to prove the Theorem (or one can  proceed as above).
\item 
This result was contained in a preliminary version of \cite{MR4518761}, where it was proved by  exploiting the homological consequences of the results of that paper. 
\end{enumerate}
\end{remark}

\subsection{The structure of $\Lambda^t (\pfin)$}
\label{subsect:structure_Lambda_pfin}

In this section we consider the projective $\Lambda^t (\pfin)$ and its subobject $\Lambda^t (\pbar)$ for $0<t \in \nat$. This is related to the analysis given by Wiltshire-Gordon \cite{2014arXiv1406.0786W} (who does not use  $\pbar$). 

\begin{conv}
\label{conv:Lambda_pbar_pfin}
The functor $\Lambda^0 (\pfin)$ is  the constant functor $\kring$,  $\Lambda^0 (\pbar)$ is the subfunctor $\kbar$, and $\Lambda^{-1} (\pbar)$ is taken to be $\kring_\mathbf{0}$.
\end{conv}

We have the following:

\begin{lemma}
\label{lem:connectivity_Lambda}
For $0< t \in \nat$, 
\begin{enumerate}
\item 
$\Lambda^t (\pfin ) (X) =0$ if $|X | < t$ and $\Lambda^t (\pfin ) (\mathbf{t}) \cong \sgn_t$ as a $\kring \sym_t$-module; 
\item 
$\Lambda^t (\pbar ) (X) =0$ if $|X | \leq  t$ and $\Lambda^t (\pbar ) (\mathbf{t+1}) \cong \sgn_{t+1}$ as a  $\kring \sym_{t+1}$-module.
\end{enumerate}
\end{lemma}

\begin{proof}
The calculation of the underlying $\kring$-vector spaces is immediate, using that $\dim \pfin (X) = |X|$ and, for $X$ a non-empty finite set, $\dim \pbar (X) = |X|-1$.

By definition, $\Lambda^t (\pfin ) (\mathbf{t}) \cong \Lambda^t( \kring \mathbf{t})$ as a $\kring \sym_t$-module, and the latter identifies with $\sgn_t$. For the case $\Lambda^t (\pbar ) (\mathbf{t+1})$, it suffices to show that, if $t>0$, the action of $\sym_{t+1}$ is non-trivial. This is checked  by considering the action of the transposition $(12)$ and using the basis of $\pbar (\mathbf{t+1})$ given by the non-zero elements of the form $[i] - [t+1]$. 
\end{proof}

\begin{proposition}
\label{prop:Lambda_pbar_simple}
For $t \in \nat$, the functor $\Lambda^t (\pbar)$ is simple. 
\end{proposition}

\begin{proof}
For $t \in \nat$, by convention, $\Lambda^0 (\pbar) = \kbar$, which is simple. Henceforth we suppose that $t>0$. By construction, $\Lambda^t (\pbar)$ is a direct summand of $\pbar^{\otimes t}$. Hence, any composition factor of $\Lambda^t (\pbar)$ is also one of $\pbar^{\otimes t}$. Now, Corollary \ref{cor:composition_series_pbar_otimes_t} implies that any such composition factor is non-zero when evaluated on $\mathbf{t+1}$. Since $\dim \Lambda^t (\pbar ) (\mathbf{t+1})= 1$ by Lemma \ref{lem:connectivity_Lambda}, it follows that there is only one such. 
\end{proof}

We record the fact that, for $0< t \in \nat$, $\Lambda^t (\pfin)$ is a direct summand of $\pfin \otimes (\pbar)^{\otimes t-1}$ (this is related to Proposition \ref{prop:refine_surject_to_pfinj}). This uses the surjection $\pfin_\mathbf{t} \twoheadrightarrow \Lambda^t (\pfin)$ that,  evaluated on a finite set $X$, is given   for $(x_i) \in X^{\times t}$ by 
\[
[x_1] \otimes \ldots \otimes [x_t]
\mapsto 
[x_1] \wedge \ldots \wedge [x_t].
\]

\begin{lemma}
\label{lem:lambda_pfin}
For $t>0$, the restriction of the surjection $\pfin_{\mathbf{t}} \twoheadrightarrow \Lambda^t (\pfin)$  to $\pfin \otimes (\pbar)^{\otimes t-1} \subset \pfin_{\mathbf{t}}$ gives a surjection 
 $ 
\pfin \otimes (\pbar)^{\otimes t-1} 
\twoheadrightarrow 
\Lambda^t (\pfin).
$ 
Hence $\Lambda^t (\pfin)$ is a direct summand of $\pfin \otimes (\pbar)^{\otimes t-1} $.
\end{lemma}

\begin{proof}
For the surjectivity, it suffices to observe that the image of 
$ 
[x_1] \otimes ([x_2] - [x_1]) \otimes  \ldots \otimes ([x_t]- [x_1])  
$ 
is $[x_1] \wedge \ldots \wedge [x_t]$. The final statement holds since $\Lambda^t (\pfin)$ is projective.
\end{proof}

We now describe the structure of $\Lambda^t (\pfin)$.

\begin{thm}
\label{thm:projective_cover_Lambda}
For $t \in \nat$, $\Lambda^t (\pfin)$ is the projective cover of the simple functor $\Lambda^{t-1} (\pbar)$ and there is a short  exact sequence:
\begin{eqnarray}
\label{eqn:ses_Lambda_pfin}
0
\rightarrow 
\Lambda^t(\pbar) 
\rightarrow 
\Lambda^t (\pfin) 
\rightarrow 
\Lambda^{t-1} (\pbar)
\rightarrow 
0,
\end{eqnarray}
where, for $t=0$, this is taken to be
 $0\rightarrow \kbar \rightarrow \kring \rightarrow \kring_{\mathbf{0}} \rightarrow 0$.
\end{thm}

\begin{proof}
By definition, one has the short exact sequence $0 \rightarrow \pbar \rightarrow \pfin \rightarrow \kbar \rightarrow 0$, which can be considered as  a two-stage filtration of $\pfin$, corresponding to a composition series. For $t>0$, on applying $\Lambda^t (-)$ this yields the short exact sequence of the statement. (At the level of vector spaces, this can be understood using the isomorphism $\Lambda^t( V \oplus \kring) \cong \Lambda^t (V) \oplus \Lambda^{t-1} (V)$.) The case $t=0$ is taken to be the given sequence, by convention.

To show that this exhibits $\Lambda^t (\pfin)$ as the projective cover of the simple functor $\Lambda^{t-1} (\pbar)$, it suffices to show that $\Lambda^t (\pfin)$ is indecomposable
. For $t=0$, using Convention \ref{conv:Lambda_pbar_pfin}, this is Lemma \ref{lem:non-split}, hence we suppose that $t>0$.  It suffices to show that $\mathrm{End}_{\kring \fin} (\Lambda^t (\pfin)) = \kring$.
 Using the canonical  surjection $\pfin_{\mathbf{t}}  \twoheadrightarrow \Lambda^t (\pfin)$, this follows from the equality 
$\hom_{\kring \fin} (\pfin_{\mathbf{t}}, \Lambda^t (\pfin)) =\kring$ given by Yoneda's lemma and the fact that that $\Lambda^t (\kring [\mathbf{t}] ) \cong \kring$, by Lemma \ref{lem:connectivity_Lambda}. 
\end{proof}

The short exact sequences (\ref{eqn:ses_Lambda_pfin}) splice to  yield the following (a related result occurs in \cite{2014arXiv1406.0786W}):

\begin{proposition}
\label{prop:Lambda-complex}
There is an exact complex: 
\[
\ldots  \rightarrow  \Lambda ^t (\pfin) \rightarrow \Lambda^{t-1} (\pfin)
\rightarrow 
\ldots \rightarrow
\Lambda ^{1} (\pfin) \rightarrow  \Lambda ^0 (\pfin) = \kring \rightarrow \kring_{\mathbf{0}} \rightarrow 0.
\]
The  complex $\Lambda^{t \geq 0}(\pfin)$ is a minimal projective resolution of $\kring_\mathbf{0}$. 
\end{proposition}

\begin{remark}
The minimal projective resolution of $\kring_\mathbf{0}$ given by  Proposition \ref{prop:Lambda-complex} is a direct summand of the projective resolution of $\kring_\mathbf{0}$ constructed from  (\ref{eqn:proj_resolution_kbar_general}).
\end{remark}

\begin{remark}
For $t$ a positive integer, 
Wiltshire-Gordon \cite{2014arXiv1406.0786W} defines  $D_t$ to be the cokernel of the map $\Lambda^{t+1}(\pfin)  \rightarrow \Lambda^t (\pfin)$ appearing in Proposition \ref{prop:Lambda-complex}. By the above identifications, one has $D_t \cong \Lambda^{t-1} (\pbar)$. 
\end{remark}

\subsection{Classification of the indecomposable projectives in $\f (\fin)$}

Putting the above results together, one has:

\begin{thm}
\label{thm:indecomposable_projectives}
Suppose that $Q$ is an indecomposable projective in $\f (\fin)$, then exactly one of the following holds
\begin{enumerate}
\item 
there exists a unique  $n \in \nat$ such that $Q \cong \Lambda^n (\pfin)$; 
\item 
there exists a unique   partition $ (\lambda \neq (1^n))  \vdash n$ for $0<n \in \nat$ such that  $Q \cong S_\lambda (\pbar)$.  
\end{enumerate}
\end{thm}

\begin{proof}
For $n \in \nat$, Theorem \ref{thm:projectives_pbar} gives that $S_\lambda (\pbar)$ is an indecomposable projective for $\lambda \neq (1^n)$ and Theorem \ref{thm:pfin_Lambda} that $\Lambda^n (\pfin)$ is an indecomposable projective. One checks using Proposition \ref{prop:decompose_tensor_pbar} that these are pairwise non-isomorphic, in the sense of the 
uniqueness in the statement. It remains to show that these exhaust the isomorphism classes of indecomposable projectives.

Suppose that $Q$ is an indecomposable projective. By Corollary \ref{cor:llproj_projective_generators}, either $Q = \kring$ or there exists $n \in \nat$ such that $Q$ is a quotient of the projective $\pfin \otimes \pbar^{\otimes n}$. Consider the second case,  for which we may suppose that $n$ is maximal. Using Lemma \ref{lem:lambda_pfin}, one can show that $\pfin \otimes \pbar^{\otimes n}$ decomposes as
\[
\Lambda^{n+1} (\pfin) \oplus (\pbar^{\otimes n+1} / \Lambda ^{n+1} (\pbar) ) \oplus (\pbar^{\otimes n} / \Lambda ^{n} (\pbar) ).
\] 
The maximality of $n$ implies that either $Q$ is isomorphic to $\Lambda^{n+1} (\pfin) $ or $Q$ is a quotient of $\pbar^{\otimes n} / \Lambda ^{n} (\pbar) $ (these possibilities are mutually exclusive). In the second case, we must have $n>0$, since $\pbar^{\otimes n} / \Lambda ^{n} (\pbar) $ is zero for $n=0$; then,  using the splitting of $\pbar^{\otimes n} / \Lambda ^{n} (\pbar)$ into isotypical components,  $Q$ must be isomorphic to $S_\lambda (\pbar)$, for a unique $\lambda \vdash n$ such that $\lambda \neq (1^{n})$. 
\end{proof}

\begin{remark}
This theorem provides an abstract classification of the simple functors in $\f (\fin)$.  Such a  classification  was given by Wiltshire-Gordon in \cite{2014arXiv1406.0786W}, using a slightly different method. The simple functors are made explicit in Section \ref{sect:simples} below.
\end{remark}

\section{The simple $\kring\fin$-modules}
\label{sect:simples}

In this section, the simple objects of $\f (\fin)$ (aka. $\kring \fin$-modules) are made explicit.  These are related to the structure of $\kring \finj (-,-)$, thus relating to results of Feigin and Shoikhet \cite{MR2350124}. 

\subsection{The  simple functors}

For $n \in \nat$,  the surjection $\pfin_\n \twoheadrightarrow \kring \finj (\n, -)$ of Proposition \ref{prop:pfinj_fin-module} restricts to a morphism that is $\sym_n\op$-equivariant  
\begin{eqnarray}
\label{eqn:pbar_tensor_pfinj}
\pbar^{\otimes n}
\rightarrow 
\kring \finj (\n, -).
\end{eqnarray}
This is neither injective nor surjective in general. The failure of surjectivity is as small as it could possibly be; it is controlled by the following Proposition:

\begin{proposition}
\label{prop:almost_surjectivity}
For $n\in \nat$, the morphism (\ref{eqn:pbar_tensor_pfinj}) fits into the exact sequence of $\kring \fin \boxtimes \kring \sym_n\op$-modules
\[
\pbar^{\otimes n}
\rightarrow 
\kring \finj (\n, -) 
\rightarrow 
\Lambda^{n-1} (\pbar) \boxtimes \sgn_n 
\rightarrow 
0,
\]
where $\Lambda^{-1} (\pbar) = \kring_\mathbf{0}$ and $\Lambda^0(\pbar) = \kbar$, by convention. 
\end{proposition}

\begin{proof}
For $n=0$, $\pbar^{\otimes 0}$ is $\kbar$ and $\Lambda^{-1} (\pbar) = \kring_\mathbf{0}$, by convention. The exact sequence in this case corresponds to the  short exact sequence (\ref{eqn:ses_k}). 

Consider the case $n>0$; Proposition \ref{prop:refine_surject_to_pfinj}  exhibits $\kring \finj (\n, -)$ as a quotient of $\pfin \otimes \pbar^{\otimes n-1}$. As in the proof of Theorem \ref{thm:indecomposable_projectives},  the latter decomposes as $\Lambda^n(\pfin) \oplus (\pbar^{\otimes n}/ \Lambda^n (\pbar)) \oplus (\pbar^{\otimes n-1} / \Lambda^{n-1}(\pbar))$, in which each summand is projective. Moreover, $\pbar^{\otimes n-1} / \Lambda^{n-1}(\pbar)$ is a direct summand of $\pfin_{\mathbf{n-1}}$. Since $\kring \finj (\n, -) $ is zero when evaluated on $\mathbf{n-1}$, Yoneda's lemma implies that there are no non-trivial maps from $\pbar^{\otimes n-1} / \Lambda^{n-1}(\pbar)$ to $\kring \finj (\n, -) $. Hence the surjection of Proposition \ref{prop:refine_surject_to_pfinj}  restricts to   a surjection 
\[
\Lambda^n(\pfin) \oplus (\pbar^{\otimes n}/ \Lambda^n (\pbar))
\twoheadrightarrow 
\kring \finj (\n, -) 
\]
with the property that the restriction to $\pbar^{\otimes n} \subset\Lambda^n(\pfin) \oplus (\pbar^{\otimes n}/ \Lambda^n (\pbar)) $ is  the morphism (\ref{eqn:pbar_tensor_pfinj}).

To conclude, it suffices to check that this restriction is not surjective. This is seen by evaluating on $\mathbf{n}$, as follows. We claim that the map $\pbar^{\otimes n} (\n) \rightarrow \kring \sym_n$ maps to the kernel of the sign representation $\sgn \colon \kring \sym_n \rightarrow \kring$. 

If $n=1$, the result is immediate, since $\pbar(\mathbf{1})=0$; for $n>1$, $\pbar(\n)$ has dimension $n-1$, with basis $\{ [i] - [n] \mid 1 \leq i \leq n-1 \}$, by Lemma \ref{lem:basis_pbar}. Thus $\pbar^{\otimes n}(\n)$ has basis given by elements of the form $\bigotimes_{j=1}^n ([i_j] - [n])$, where $i_* : \n \rightarrow \mathbf{n-1}$.
It is straightforward to check that, if $i_*$ is not surjective, then the associated basis element maps to zero. Hence suppose that $i_*$ is surjective; thus there exist $j_1 < j_2$ such that  $i_{j_1}= i_{j_2}$. To illustrate the argument, consider the case $n=2$, so that $j_1=1$ and $j_2=2$ and one has the (unique) basis element $([1]-[2]) \otimes ([1]-[2])$; this maps to $-[\id] - [\tau] \in \kring \sym_2$, which lies in the kernel of the sign representation. The general case follows by extending this analysis.
\end{proof}

In the following, $\kring \finj( \n, -)$ is considered as a $\kring \fin$-module via Proposition \ref{prop:pfinj_fin-module}.

\begin{definition}
\label{defn:identify_C_lambda}
For $0< n \in \nat$ and a partition $\lambda \vdash n$, define 
\[
C_\lambda := 
\left\{
\begin{array}{ll}
\kring \finj (\n, - ) \otimes_{\sym_n} S_\lambda & \lambda \neq (1^n) \\
\Lambda^n (\pbar) & \lambda = 1^n.
\end{array}
\right.
\]
For $n=0$, take $C_{(0)}:= \kbar$.
\end{definition}

Proposition \ref{prop:almost_surjectivity} gives:

\begin{corollary}
\label{cor:C_lambda_quotient}
For $0< n \in \nat$ and a partition $\lambda \vdash n$, the morphism (\ref{eqn:pbar_tensor_pfinj}) induces a surjection 
\[
S_\lambda (\pbar ) \twoheadrightarrow C_\lambda.
\]
In particular, $C_\lambda$ is a quotient of $\pbar^{\otimes n}$.
\end{corollary}

For some purposes it is useful to reindex the functors corresponding to partitions of the form $(1^t)$ as follows:

\begin{definition}
\label{defn:identify_Ctilde_lambda}
For $n \in \nat$ and a partition $\lambda \vdash n$, define 
\[
\tilde{C}_\lambda := 
\left\{
\begin{array}{ll}
C_\lambda  & \lambda \neq (1^n) \\
\Lambda^{n-1} (\pbar) & \lambda = 1^n,
\end{array}
\right.
\]
where $\Lambda^{n-1} (\pbar)$ is understood to be $\kring_\mathbf{0}$ for $n=0$ and $\kbar$ for $n=1$.
\end{definition}

By restriction along $\kring \fb \hookrightarrow \kring \fin$,  any $\kring \fin$-module  $F$ has an underlying $\kring \fb$-module structure; in particular, for $t \in \nat$, $F(\mathbf{t})$ is a $\kring \sym_t$-module. The following follows directly from the definition together with Lemma \ref{lem:connectivity_Lambda}.

\begin{lemma}
\label{lem:evaluate_C_lambda}
For $n \in \nat$ and a partition $\lambda \vdash n$, for $t \leq n $ there are isomorphisms of $\sym_t$-modules:
\[
\tilde{C}_\lambda (\mathbf{t}) = 
\left\{
\begin{array}{ll}
0 & t < n \\
S_\lambda & t = n.
\end{array}
\right.
\]
\end{lemma}

The classification of the simple functors is most cleanly stated using the family $\{ \tilde{C}_\lambda \mid \lambda \vdash n\}$:

\begin{thm}
\label{thm:classify_simples}
For $n \in \nat$ and $\lambda \vdash n$, the $\kring\fin$-module $\tilde{C}_\lambda$ is simple. Moreover, 
\begin{enumerate}
\item 
if $\lambda \neq (1^n)$, $S_\lambda (\pbar) \rightarrow \tilde{C}_\lambda$ is the projective cover of $\tilde{C}_\lambda = C_\lambda $; 
\item 
$\Lambda^n (\pfin) \twoheadrightarrow \Lambda^{n-1}(\pbar) = \tilde{C}_{(1^n)}$ is the projective cover of $\tilde{C}_{(1^n)}$.
\end{enumerate}

For $F$  a simple $\kring\fin$-module, such that $F(\mathbf{t})=0$ for $t <n $  and  $F(\n)\neq 0$,
  $F(\n)$ is a simple $\sym_n$-module and 
$
F \cong \tilde{C}_\lambda
$
 where $F(\n) \cong S_\lambda$. (If $F(\mathbf{0})\neq 0$, this corresponds to $F \cong \kring_\mathbf{0}$.)
\end{thm}

\begin{proof}
We first prove that, for $\lambda \neq (1^n)$,  $C_\lambda = \tilde{C}_\lambda$ is simple. Suppose otherwise, since $C_\lambda (\n)$ is a simple $\sym_n$-module by Lemma \ref{lem:evaluate_C_lambda},  then $C_\lambda$ would have a non-zero composition factor $G$ such that $G(\mathbf{n}) =0$. 
 Since $C_\lambda$ is a quotient of $\pbar^{\otimes n}$ by Corollary \ref{cor:C_lambda_quotient}, $G$ is a composition factor of the latter, hence Corollary \ref{cor:composition_series_pbar_otimes_t}  gives that $G(\mathbf{n+1})\neq 0$.
 
Now, our understanding of the indecomposable projective generators of $\f (\fin)$ (see Theorem \ref{thm:indecomposable_projectives}) implies that either $G$ is a quotient of $S_\mu (\pbar)$ for some partition $\mu \vdash n+1$  or $G$ is $\Lambda^n (\pbar)$, hence a quotient of $\Lambda^{n+1}(\pfin)$. The first case is excluded since $\hom_{\kring\fin} (S_\mu (\pbar), \pbar^{\otimes n})=0$, as a consequence of Corollary \ref{cor:morphisms_pbar}, since $S_\mu (\pbar)$ is a direct summand of $\pbar^{\otimes n+1}$. The second case is excluded similarly, since $\hom_{\kring\fin} (\Lambda^{n+1}(\pfin) , S_\lambda (\pbar))=0$ for $\lambda \neq (1^n)$, again as a consequence of Corollary \ref{cor:morphisms_pbar}; explicitly, any such morphism would factor across $\Lambda^n (\pbar)$ and $\hom_{\kring\fin} (\Lambda^{n}(\pbar) , S_\lambda (\pbar))=0$ by Proposition \ref{prop:decompose_tensor_pbar}.

The remaining statements follow, using the indecomposability of the projectives of the form $S_\lambda (\pbar)$ and $\Lambda^n (\pfin)$, as in  Theorem \ref{thm:indecomposable_projectives}.
\end{proof}

As a consequence, one sees that $\kring \finj (\n, -)$ is almost semi-simple: 

\begin{corollary}
\label{cor:structure_pfinj}
For $0< n \in \nat$, there is an isomorphism of $\kring\fin$-modules
\[
\kring \finj (\n, -) 
\cong 
\Lambda^n (\pfin) 
\oplus 
\bigoplus_{\substack{\lambda \vdash n \\ \lambda \neq (1^n) }}
C_\lambda^{\oplus \dim S_\lambda}.
\]
\end{corollary}

\begin{remark}
In \cite{2022arXiv221000399S}, Sam, Snowden and Tosteson  gave a model for the category of $\kring\fin\op$-modules using polynomial representations of the Witt Lie algebra. They observed that, under their equivalence (and using duality), the classification of the simple $\kring\fin$-modules given in \cite{2014arXiv1406.0786W} corresponds to the classification of certain irreducible modules over the Witt Lie algebras given by Feigin and Shoikhet \cite{MR2350124}. 

The construction of the irreducible objects given by Feigin and Shoikhet can be understood as the analysis of the structure of 
$\kring \finj (\n, -) $ as a $\kring\fin$-module (cf. Corollary \ref{cor:structure_pfinj}).
\end{remark}

\begin{remark}
\label{rem:projectives_finj_versus_simples_fin}
Restricting to the $\kring\finj$-module structure, for any partition $\lambda \vdash n$, $\kring \finj (\n, - ) \otimes_{\sym_n} S_\lambda$ is the projective cover of the simple $\kring \finj$-module $S_\lambda$ supported on $\n$. This projective does not have finite length, whereas (for $\lambda \neq (1^n)$) it is the restriction to $\kring \finj$ of a {\em simple} $\kring \fin$-module!
\end{remark}

\subsection{The image of the norm-like map}

Proposition \ref{prop:nat_finj_fsop} provides the natural transformation of $\kring\fin$-modules
\begin{eqnarray}
\label{eqn:finj_fsop_n}
\kring \finj (\n, -) \rightarrow D \fs (-, \n).
\end{eqnarray}

Using Corollary \ref{cor:structure_pfinj} it is straightforward to determine the kernel of this map, as follows.

\begin{example}
For $n=1$, one has $\finj (\mathbf{1}, X) \cong X$ and $\fs (X, \mathbf{1}) = \{*\}$. The map (\ref{eqn:finj_fsop_n}) identifies as the map  $\kring [X] \rightarrow \kring$ of (\ref{eqn:proj_kX_k}). By definition, this has kernel $\pbar (X)$. 
\end{example}

\begin{proposition}
\label{prop:ker_norm_map}
For $0< n \in \nat$, the map  (\ref{eqn:finj_fsop_n}) has kernel $\Lambda^n (\pbar)$.
\end{proposition}

\begin{proof}
By Corollary \ref{cor:structure_pfinj}, all composition factors of $\kring \finj (\n, -)$ are detected by evaluation on $\n$, apart from $\Lambda^n (\pbar)$, which is detected  on $\mathbf{n+1}$. Now, Proposition \ref{prop:nat_finj_fsop} gives that the map is an isomorphism when evaluated upon $\n$. It follows that the kernel of the map is contained in $\Lambda^n (\pbar)$. Hence, to conclude, it suffices to show that the map is {\em not} injective when evaluated upon $\mathbf{n+1}$. 

Considered as a $\sym_{n+1}$-module, $\Lambda^n (\pbar) (\mathbf{n+1})$ identifies as the sign representation $\mathrm{sgn}_{n+1}$, by Lemma \ref{lem:connectivity_Lambda}. To conclude, one uses that $\kring \fs (\mathbf{n+1}, \mathbf{n})$, considered as a $\sym_{n+1}$-module, does not contain   $\mathrm{sgn}_{n+1}$, which is straightforward.
\end{proof}

\begin{corollary}
\label{cor:factor_norm}
For $0<n \in \nat$, the natural transformation  (\ref{eqn:norm_fin}) for $X= \n$  factors as 
\[
\kring \fin (\n,-) \twoheadrightarrow 
\kring \finj (\n, -) /\Lambda^n (\pbar)
\hookrightarrow 
D \kring \fin (-,\n).
\]  
Here, $\kring \finj (\n, -) /\Lambda^n (\pbar) \cong \bigoplus_{\substack{\lambda \vdash n }}
\tilde{C}_\lambda^{\oplus \dim S_\lambda}$ is semi-simple.
\end{corollary}

\subsection{Evaluation of the simples}

The purpose of this section is to generalize Lemma \ref{lem:evaluate_C_lambda} by identifying $C_\lambda (\mathbf{t})$ as a $\kring \sym_t$-module, for each simple functor $C_\lambda$. 

\begin{notation}
\ 
\begin{enumerate}
\item 
For partitions $\lambda$, $\mu$ write $\mu \preceq \lambda$ if $\mu_i \leq \lambda_i$ whenever $\mu_i$ is defined (this can be pictured as stating that the Young diagram of $\mu$ is contained within that of $\lambda$). If $\mu \preceq \lambda$, the associated skew partition is denoted $\lambda / \mu$. 
\item 
If $\mu \preceq \lambda$, write $\lambda / \mu \in \hs$ to signify that $\lambda / \mu$ is a horizontal strip, i.e., contains at most one box in each column of the associated Young diagram; otherwise  $\lambda / \mu \not \in \hs$.
\end{enumerate}
\end{notation}

\begin{proposition}
\label{prop:evaluate_simples}
Suppose that $\lambda \vdash n$ is a partition of a positive integer $n$  and let $t \in \nat$. 
\begin{enumerate}
\item 
If $\lambda \neq (1^n)$, then as a $\kring \sym_t$-module
\[
C_\lambda (\mathbf{t}) \cong  
\left\{
\begin{array}{ll}
0 & t < n \\
\bigoplus_{\substack{\lambda \preceq \mu \\ |\mu |=t \\ \mu/\lambda \in \hs}} S_\mu & t \geq n.
\end{array}
\right.
\]
\item 
If $\lambda = (1^n)$, so that $C_\lambda = \Lambda^n (\pbar)$, as a $\kring \sym_t$-module
\[
\Lambda ^n(\pbar) (\mathbf{t}) \cong 
\left\{
\begin{array}{ll}
0 & t \leq n \\
S_{(t-n, 1^n)} & t > n.
\end{array}
\right.
\]
\end{enumerate}
\end{proposition}

\begin{proof}
In the first case, $C_\lambda (\mathbf{t}) \cong \kring \finj (\mathbf{n}, \mathbf{t}) \otimes _{\sym_n} S_\lambda$ as a $\kring \sym_t$-module. This is zero if $t<n$. If $t \geq n$, the permutation $\kring (\sym_n\op \times \sym_t)$-module $\kring \finj (\mathbf{n}, \mathbf{t})$ is isomorphic to $\kring \sym_t / \sym_{t-n} \cong \kring \sym_t \otimes_{\sym_{t-n}} \triv_{t-n}$, using the residual right $\sym_n$-action (cf. the proof of Proposition \ref{prop:identify_induction}). Thus one has the isomorphism of $\kring \sym_t$-modules 
\[
\kring \finj (\mathbf{n}, \mathbf{t}) \otimes _{\sym_n} S_\lambda
\cong
\kring \sym_t \otimes_{\sym_{t-n} \times \sym_n} (\triv_{t-n} \boxtimes S_\lambda).
\]
The right hand side identifies with the given expression by the Pieri rule. 

The second case is proved by a similar argument. One first consider $\kring \finj (\mathbf{n}, \mathbf{t}) \otimes _{\sym_n} \sgn_n$. As above, we need only consider the case $t \geq n$.  The only partitions $\mu \vdash t$ for which $(1^n) \preceq \mu$ and $\mu / (1^n) \in \hs$ are $(t-n,1^n)$ (if $t>n$)  and $(t-(n-1), 1^{n-1})$. Now, by Corollary \ref{cor:structure_pfinj}, the functor $\kring \finj (\mathbf{n}, -) \otimes _{\sym_n} \sgn_n$ is isomorphic to $\Lambda^n(\pfin)$, which has two composition factors, namely $\Lambda^n (\pbar)$ and $\Lambda^{n-1} (\pbar)$. By an obvious induction, the contribution from $\Lambda^n (\pbar)(\mathbf{t})$ is zero for $t=n$ and $S_{(t-n,1^n)}$ for $t>n$. 
\end{proof}

This allows us to decompose the standard projectives $S_\lambda (\pbar)$ as a direct sum of indecomposable projectives, for $\lambda \vdash n$, with $n$ a positive integer.

\begin{thm}
\label{thm:decompose_Slambda_Pfin}
For $n$ a positive integer and  a partition $\lambda \vdash n$,
\[
S_\lambda (\pfin) 
\cong 
\left\{
\begin{array}{ll}
\bigoplus_{\substack{\nu \preceq \lambda \\ \lambda/\nu \in \hs}}
S_\nu (\pbar) & \lambda \neq (s ,1^{n-s}) \mbox{ for $1 \leq s \leq n$} 
\\
 \Lambda^{n-s+1}(\pfin)  \oplus 
\bigoplus_{\substack{\nu \preceq \lambda \\ \nu \not \in \{ (1^{n-s+1}), (1^{n-s}) \}\\ \lambda/\nu \in \hs}}
S_\nu (\pbar)
 & \lambda = (s ,1^{n-s}) \mbox{ for $1 \leq s \leq n$}.
\end{array}
\right.
\]
\end{thm}

\begin{proof}
By standard arguments, the multiplicity of the projective cover of the simple functor $C_\nu$ as a direct summand of $S_\lambda (\pfin)$ is 
\[
\dim \hom_{\kring \fin} (S_\lambda (\pfin), C_\nu) = \dim \hom_{\sym_n} (S_\lambda, C_\nu (\mathbf{n})).
\] 
The right hand side is calculated by Proposition \ref{prop:evaluate_simples} (this is either $0$ or $1$, depending on the value of $\nu$). Since $n$ is positive, if this is non-zero  we must have $|\nu|>0$. This yields the decomposition in the statement, using the identification of the projective covers given in Theorem \ref{thm:classify_simples}.
\end{proof}

\section{The projective cover of $\pbar^{\otimes \bullet}$ and a Morita equivalence}
\label{sect:proj_cover}

The projective covers of the $\kring \fin$-modules $\pbar^{\otimes n}$, for $ n \in \nat$,  together with $\pfin_\mathbf{0} = \kring$, provide a convenient choice of projective generators for the category $\f (\fin)$. By Morita equivalence, one obtains an alternative description of $\f (\fin)$ (see Theorem \ref{thm:fin-modules_morita}). 

\subsection{Introducing the projective cover of $\pbar^{\otimes n}$}

Recall the convention $\pbar^{\otimes 0} = \kbar$ and define $\pbar^{\otimes -1}$ to be $\kring_\mathbf{0}$. Projective covers of these functors are described by the following short exact sequences 
\begin{eqnarray*}
&&0
\rightarrow 
\kbar 
\rightarrow 
\kring = \pfin_{\mathbf{0}} 
\rightarrow 
\kring_\mathbf{0} = \pbar^{\otimes -1}
\rightarrow 
0
\\
&& 0
\rightarrow 
\pbar 
\rightarrow 
\pfin 
\rightarrow 
\kbar = \pbar^{\otimes 0}
\rightarrow 
0.
\end{eqnarray*}

Extending these examples, one can take the following definition:

\begin{definition}
\label{defn:proj_cover}
For $n \in \nat \cup \{-1 \}$, let $\mathbb{P}_n$ be the projective cover of $\pbar^{\otimes n}$ given by 
\[
\mathbb{P}_n := \Lambda ^{n+1} (\pfin) \oplus (\pbar^{\otimes n} / \Lambda^n (\pbar) )
\]
equipped with the  surjection $\mathfrak{p}_n : \mathbb{P}_n \twoheadrightarrow \pbar^{\otimes n}$ induced by 
the projection $\Lambda^{n+1} (\pfin) \twoheadrightarrow \Lambda^n (\pbar) \subset \pbar^{\otimes n}$ and the (canonical) inclusion of the direct summand $(\pbar^{\otimes n} / \Lambda^n (\pbar) )$ in $\pbar^{\otimes n}$.
\end{definition}

By construction, one has the short exact sequence 
\[
0
\rightarrow 
\Lambda^{n+1} (\pbar) 
\rightarrow 
\mathbb{P}_n 
\stackrel{\mathfrak{p}_n}{\rightarrow} 
\pbar^{\otimes n} 
\rightarrow 
0.
\]
(That $\mathbb{P}_n$ is a projective cover of $\pbar^{\otimes n}$ is clear.)

\begin{lemma}
\label{lem:morphisms_mathfrak_p}
For $n \in \nat \cup \{-1\}$, $\mathfrak{p}_n$ induces  isomorphisms
\begin{eqnarray*}
\hom_{\kring \fin} (\pbar^{\otimes n}, \pbar ^{\otimes n} ) 
&\stackrel{\cong}{\rightarrow}& 
\hom_{\kring\fin} (\mathbb{P}_n, \pbar ^{\otimes n} )
\\
\hom_{\kring \fin} (\mathbb{P}_n, \mathbb{P}_n )
&\stackrel{\cong}{\rightarrow}& 
\hom_{\kring\fin} (\mathbb{P}_n, \pbar ^{\otimes n} ).
\end{eqnarray*}
\end{lemma}

\begin{proof}
The first statement follows from the fact that $\hom _{\kring \fin} (\Lambda^{n+1}(\pbar), \pbar ^{\otimes n} ) =0$, a consequence of Corollary \ref{cor:morphisms_pbar}. The second follows from the projectivity of $\mathbb{P}_n$ together with the fact that $\hom_{\kring \fin} (\mathbb{P}_n, \Lambda^{n+1} (\pbar))=0$, since $\Lambda^{n+1}(\pbar)$ is not a composition factor of the head of $\mathbb{P}_n$ (the largest semi-simple quotient of $\mathbb{P}_n$, sometimes called the cosocle). 
This is seen as follows: by construction the head  of $\mathbb{P}_n$ identifies with that of $\pbar^{\otimes n}$; the latter does not contain $\Lambda^{n+1}(\pbar)$ as a composition factor.
\end{proof}

As a consequence one has: 

\begin{proposition}
\label{prop:endomorphism_proj_cover}
For $n \in \nat \cup \{-1 \}$, the morphism $\mathfrak{p}_n$ induces an isomorphism  $\mathrm{End}_{\kring \fin} (\mathbb{P}_n) \cong  \mathrm{End}_{\kring \fin} (\pbar^{\otimes n})$ of $\kring$-algebras. In particular,  $\mathrm{End}_{\kring \fin} (\mathbb{P}_n) \cong \kring \sym_n$ as $\kring$-algebras, where $\sym_{-1}$ is taken to be the trivial group, by convention.
\end{proposition}

\begin{proof}
Lemma \ref{lem:morphisms_mathfrak_p} gives the isomorphism of $\kring$-vector spaces  $\mathrm{End}_{\kring \fin} (\mathbb{P}_n) \cong  \mathrm{End}_{\kring \fin} (\pbar^{\otimes n})$. It remains to check that this is an isomorphism of $\kring$-algebras; this is straightforward.
\end{proof}

\begin{remark}
\label{rem:right_actions}
The canonical action of $\kring \sym_n$ on $\pfin_\n$ is on the right. To be consistent with this via the maps 
\[
\mathbb{P}_n \twoheadrightarrow \pbar^{\otimes n} \hookrightarrow \pfin_\n,
\]
the action of $\sym_n$ on $\mathbb{P}_n$ and on $\pbar^{\otimes n}$ is taken to be on the right. For $\pbar^{\otimes n}$ this is the action on the right by permutation of tensor factors and that on $\mathbb{P}_n$ is given by Proposition \ref{prop:endomorphism_proj_cover} (as a right action).
\end{remark}

\subsection{Projective generators and the associated Morita equivalence}

The following essentially follows from the constructions:

\begin{proposition}
\label{prop:projec_generators}
The set $\{ \mathbb{P}_n \mid  n \in \nat \cup \{-1 \} \}$ is a set of projective generators  of $\f (\fin)$.
\end{proposition}

\begin{proof}
The indecomposable projectives in $\f(\fin)$ are classified in Theorem \ref{thm:indecomposable_projectives}. It is clear that any such indecomposable projective occurs as a direct summand of $\mathbb{P}_n$ for some $n \in \nat \cup \{-1 \}$. 
\end{proof}

This has the  consequence:

\begin{corollary}
\label{cor:multiplicity_composition_factors}
For a $\kring \fin$-module $F$, 
\begin{enumerate}
\item 
the multiplicity of $\kring _\mathbf{0}$ as a composition factor of $F$ is $\dim F(\mathbf{0})$;
\item 
for $\lambda \vdash n$ with $n>0$, the multiplicity of $C_\lambda$ as a composition factor of $F$ is equal to 
\[
\dim \hom_{\sym_n} (S_\lambda ,\hom_{\kring \fin} (\mathbb{P}_n , F)),
\]
using the $\kring \sym_n$-module structure on $\hom_{\kring \fin} (\mathbb{P}_n , F)$ induced by the $\kring \sym_n\op$-action on  $\mathbb{P}_n$.
\end{enumerate}
In particular, $\hom_{\kring \fin}(\mathbb{P}_\bullet, F)$, considered as a functor on $\fb \amalg \{-1\}$,   determines the composition factors of $F$.
\end{corollary}

\begin{remark}
\label{rem:duality_hom_tensor}
For $\lambda \vdash n$ with $n>0$, since the simple $\sym_n$-modules are self dual and invariants and coinvariants are isomorphic in characteristic zero, there is an isomorphism of vector spaces
\[
\hom_{\sym_n} (S_\lambda ,\hom_{\kring \fin} (\mathbb{P}_n , F))
\cong 
S_\lambda \otimes_{\sym_n}\hom_{\kring \fin} (\mathbb{P}_n , F) 
\]
where, on the right, $S_\lambda$ is considered as a right $\sym_n$-module. Such identifications are used henceforth without further comment.
\end{remark}

\begin{example}
\label{exam:multiplicity_Lambda}
For $0< n \in \nat$, the multiplicity of $\Lambda^n (\pbar)$ in $F$ is the dimension of 
\[
\hom_{\kring \fin} (\Lambda^{n+1}(\pfin) , F) \cong  \sgn_{n+1} \otimes_{\sym_{n+1}} F(\mathbf{n+1}).
\]
\end{example}

\begin{notation}
\label{nota:calc_proj_cover}
Let $\calc (\mathbb{P}_\bullet)$ denote the full subcategory of $\f (\fin)$ with set of objects $\{ \mathbb{P}_n \mid n \in \nat \cup \{-1 \} \} $.
\end{notation}

Since $\{ \mathbb{P}_n \mid n \in \nat \cup \{-1 \} \} $ is a set of small projective generators of $\f (\fin)$,  Freyd's theorem (see \cite[Theorem 3.1]{MR294454} for a version over $\zed$) gives:

\begin{thm}
\label{thm:fin-modules_morita} 
The category  $\f (\fin)$ is equivalent to the category of $\calc(\mathbb{P}_\bullet)\op$-modules (i.e., $\kring$-linear functors from $\calc (\mathbb{P}_\bullet)\op$ to $\kmod$).
\end{thm}

\part{Calculations in characteristic zero}
\label{part:calculate}

In this part we implement the  results of Part \ref{part:zero}. In particular, Theorem \ref{thm:hom_proj_cover}, which shows how to calculate the composition factors of a $\kring \fin$-module $F$, is applied notably to the cases $F=  \pbar^{\otimes n}$ and its projective cover. 
 These results are applied in Theorem \ref{thm:endo_pbar_otimes} to calculate the space of morphisms between the $\kring \fin$-modules $\pbar^{\otimes m}$ and $\pbar^{\otimes n}$, for $m,n \in \nat$. This extends  the partial calculation Corollary \ref{cor:morphisms_pbar}.

The results exploit general inversion techniques for suitable $\kring \fb$-modules that are introduced in Section \ref{sect:invert}.

\section{Inverting $-\odot \triv$}
\label{sect:invert}

Working in the category of $\kring \fb$-modules,  one has the functor $- \odot \triv$, where $\triv$ is as in Notation \ref{nota:triv}.  The purpose of this section is to explain that this functor is invertible in a suitable sense (see Proposition \ref{prop:invert_odot_triv} for the basic statement). This also applies when working with $\kring \fb\op$-modules, {\em mutatis mutandis}.

One application arises from Example \ref{exam:kring_fin_odot}, which shows that the underlying $\kring \fb$-bimodule of $\kring \fin$ is isomorphic to $\triv \odot_\fb \kring \fs$. Thus inversion allows  the underlying $\kring \fb$-bimodule of $\kring \fs$ to be recovered from that of $\kring \fin$.

\begin{remark}
This material could be explained by working with characters, where such techniques are standard. However, in the presence of more structure,  working with the Grothendieck group is preferable. 
\end{remark}

\subsection{The basic result}

We require to work with isomorphism classes of objects and  with virtual objects; to do so we pass to the Grothendieck group of $\kring \fb$-modules. This is possible, since all the $\kring \fb$-modules considered here have the property that, evaluated on a finite set, they have finite dimension. Such technical details are suppressed, for simplicity of exposition. 

\begin{notation}
\label{nota:groth}
\ 
\begin{enumerate}
\item 
For $M$ a $\kring \fb$-module, write $[M]$ for its class in the Grothendieck group. 
\item 
For $M, N$  $\kring \fb$-modules, denote by $[M]\odot [N]$ the class $[M\odot N]$. This extends bilinearly to define $\odot$ on the whole Grothendieck group. 
\end{enumerate}
\end{notation}

For example, one has $[\triv]$ in the Grothendieck group and the associated operation $- \odot [\triv]$, corresponding to the functor $-\odot \triv$.  For the sign representations, we work with the infinite sum 
\[
\W(0) := \sum_{t \in \nat} (-1)^t[\sgn_t].
\]

The key invertibility result is the following, in which $\kring_\mathbf{0}$ denotes $\kring$ supported on $\mathbf{0}$.

\begin{proposition}
\label{prop:invert_odot_triv}
There is an equality in the Grothendieck group of $\kring \fb$-modules:
\begin{eqnarray}
\label{eqn:invert}
[\triv] \odot \W(0) 
=
[\kring_\mathbf{0}].
\end{eqnarray}
\end{proposition}

\begin{proof}
This is a standard result; for completeness, a proof is given. Consider the evaluation on $\n$, thus passing to the Grothendieck group of $\kring \sym_n$-modules. 
When $n=0$, there is only one term that appears, namely the trivial representation. For $0<n \in \nat$,  the term on the left hand side of (\ref{exam:kfin}) yields
\[
\sum_{k=0}^n
(-1)^k [\triv_{n-k} \odot \sgn_k].
\]
By the Pieri rule, 
$ 
\triv_{n-k} \odot \sgn_k 
\cong 
S_{(n-k+1,1^{k-1})} \oplus S_{(n-k,1^k)}$, 
where the first term is understood to be zero when $k=0$ and the second to be zero when $n=k$.
A straightforward verification shows that the alternating sum in the Grothendieck group is zero.
\end{proof}

\begin{example}
\label{exam:kring_fs_from_fin}
Using the identification given in Example \ref{exam:kring_fin_odot}, one obtains the equality 
$ 
[\kring \fs] = \W(0) \odot_\fb [\kring \fin]
$ 
in the Grothendieck ring of $\kring \fb $-bimodules.
\end{example}

\subsection{A truncated variant}

\begin{notation}
\label{nota:H_S}
For $k \in \nat$, set 
\begin{enumerate}
\item
$H(0) = \kring$, considered as a $\kring \fb$-module supported on $\mathbf{0}$ and, for $0<k$, 
$$H(k) := \bigoplus_{n \geq k} S_{(n-k+1,1^{k-1})};$$
\item 
the element of the Grothendieck group of $\kring \fb$-modules 
$$\W(k):= \sum_{t \geq 0} (-1)^{t + k} [\sgn_{k+t}].$$ 
\end{enumerate}
\end{notation}

\begin{remark}
\ 
\begin{enumerate}
\item 
For $k>0$, $H(k)$ encodes the {\em hook representations} indexed by hook partitions with first column of length $k$ (which explains the choice of notation).
\item 
The $\kring \fb$-module $H(1)$ is isomorphic to $\overline{\triv}$, the submodule of $\triv$ supported on non-empty finite sets.
\end{enumerate}
\end{remark}

By inspection, one has the following:

\begin{lemma}
\label{lem:W_relations}
For $k \in \nat$,
\begin{enumerate}
\item 
$\W (k) (\n)=0$ for $n<k$ and $\W (k) (\mathbf{k}) =   [\sgn_k]$;
 \item 
$
\W (k) + \W (k+1) = [\sgn_k]$ 
 in the Grothendieck group of $\kring \fb$-modules.
\end{enumerate}
\end{lemma}

There is the following basic relation for the hook representations:

\begin{lemma}
\label{lem:H-relation}
For $k \in \nat$, one has 
$ 
H (k) \oplus H (k+1) \cong \sgn_k \odot \triv.
$ 
\end{lemma}

\begin{proof}
This is clear by inspection for $k=0$. 
 For $k>0$, by definition, 
\[
H (k) \oplus H (k+1)
=
\bigoplus_{n \geq k} \big( S_{(n-k+1,1^{k-1})} \oplus S_{(n-k, 1^k)} \big),
\]
where $S_{(n-k, 1^k)}$ is understood to be $0$ for $n=k$.  Now, for each $n \geq k$, $S_{(n-k+1,1^{k-1})} \oplus S_{(n-k, 1^k)} \cong \sgn_k \odot \triv_{n-k}$, by the Pieri rule (as in the proof of Proposition \ref{prop:invert_odot_triv}). The sum $ \bigoplus_{n \geq k} \sgn_k \odot \triv_{n-k}$ identifies with $\sgn_k \odot \triv$, by the definition of the latter.
\end{proof}

The invertibility of $- \odot \triv$ yields the following:

\begin{proposition}
\label{prop:Hook_inversion}
For $0< k \in \nat$, there is an equality in the Grothendieck group of $\kring \fb$-modules:
\[
[H(k)] \odot \W(0) 
= 
\W(k).
\]
\end{proposition}

\begin{proof}
The result is proved by induction on $k$, starting from the basic case $k=0$, which is immediate.

For the inductive step, Lemma \ref{lem:H-relation} gives 
\[
[H(k)] \odot \W(0)
\ + \ 
[H(k+1)] \odot \W(0)
= 
[\sgn_k \odot \triv] \odot \W(0) 
= 
[\sgn_k],
\]
using Proposition \ref{prop:invert_odot_triv} for the second equality, since $\odot $ is associative.
 Now, by the inductive hypothesis, we have $[H(k)] \odot \W(0) =\W(k)$, hence
\[
[H(k+1)] \odot \W(0)   = [\sgn_k] - \W(k).
\]
The right hand side is equal to $\W(k+1)$ by Lemma \ref{lem:W_relations}, as required. 
\end{proof}

\begin{remark}
The identity of Proposition \ref{prop:Hook_inversion} is equivalent to the equality $[H(k)] = \W(k) \odot [\triv]$ 
 in the Grothendieck group of $\kring \fb$-modules. This can be viewed as a `truncated' version of Proposition \ref{prop:invert_odot_triv}.
\end{remark}

\section{Calculating $\hom(\mathbb{P}_\bullet, F)$}
\label{sect:hom_mathbbP}

The purpose of this section is to show how to calculate $\hom_{\kring \fin} (\mathbb{P}_\bullet, F)$ in terms of the underlying $\kring \fb$-module of a $\kring \fin$-module $F$ that takes finite-dimensional values. This gives an explicit description of the multiplicities of the composition factors of such a $\kring \fin$-module (see Theorem \ref{thm:hom_proj_cover}).

This relies upon  Theorem \ref{thm:pfin_Lambda}, which  relates the following, considered as $\kring \fb\op \otimes \kring \fin$-modules:
\begin{eqnarray*}
\pfin_\bullet &:& \n \mapsto \pfin_\n \\
\pbar^{\otimes \bullet} &:& \n \mapsto \pbar^{\otimes n}.
\end{eqnarray*}

\subsection{Splitting off projective covers of exterior power functors}

Due to the exceptional behaviour of the simple functors $\Lambda^t (\pbar)$, for $t \in \nat$, one is lead to consider 
\begin{eqnarray*}
\pbar^{\otimes \bullet}/ \Lambda  &:& \n \mapsto \pbar^{\otimes n}/\Lambda^n (\pbar).
\end{eqnarray*}
We identify an analogous summand  of $\pfin_\n$ that eliminates the projective direct summands of the form $\Lambda^t (\pfin_\mathbf{1})$, using the following:

\begin{proposition}
\label{prop:project_Lambda_n}
For $n \in \nat$, there is a $\kring \sym_n\op$-equivariant surjection 
$ 
p_{\Lambda, \n}  :  \pfin_\n \twoheadrightarrow \pfin_{\Lambda, \n}
$ 
such that 
\begin{enumerate}
\item 
for $n=0$, $\pfin_{\Lambda, \mathbf{0}} = \pfin_\mathbf{0} = \kring$; 
\item 
for $n>0$, 
$
\pfin_{\Lambda, \n}
\cong 
\bigoplus_{l=1} ^n 
\Lambda^l (\pfin) \boxtimes S_{(n-l+1, 1^{l-1})}$,  
in particular it is projective; 
\item 
$\ker p_{\Lambda, \n}$ is projective and does not contain an indecomposable summand of the form $\Lambda^t (\pfin)$. 
\end{enumerate}
The morphism $p_{\Lambda, \n}$ is unique up to non-canonical isomorphism. 
\end{proposition}

\begin{proof}
The case $n=0$ is clear; henceforth suppose that $n>0$. The result follows using standard arguments from the fact that, for $k \in \nat$, $\Lambda^{k+1} (\pfin)$ is the projective cover of the simple functor $\Lambda^k (\pbar)$, by  Theorem \ref{thm:projective_cover_Lambda}.  Namely, for $n>0$, $\pfin_\n (\mathbf{0})$ is zero; hence $\Lambda^0 (\pfin) = \kring$ is not a direct summand of $\pfin_n$. For $l >0$, the `multiplicity' of $\Lambda^l (\pfin)$ as a direct summand of $\pfin_\n$ is given by 
$
\hom_{\kring \fin} (\pfin_\n , \Lambda^{l-1}(\pbar))$,  
which also records the $\kring \sym_n$-module structure of this component. 
The result thus follows from Proposition \ref{prop:evaluate_simples}. 
\end{proof}

\begin{notation}
Denote by 
\begin{enumerate}
\item 
$\pfin_\bullet / \Lambda $  the $\kring \fb\op \otimes \kring \fin$-module given by 
$ 
\n \mapsto 
\ker p_{\Lambda, \n}.
$ 
\item 
$\pbar^{\otimes \bullet}/ \Lambda$ the 
$\kring \fb\op \otimes \kring \fin$-module given by 
$ 
\n \mapsto 
\pbar^{\otimes n}/ (\sgn_n \boxtimes \Lambda^n (\pbar)),
$ 
where the $\sgn_n$ keeps track of the $\kring \sym_n\op$-action.
\end{enumerate}
\end{notation}

By construction, for each $n \in \nat$, one has a short exact sequence of $\kring \sym_n\op \otimes \kring \fin$-modules:
\[
0
\rightarrow 
(\pfin_\bullet / \Lambda)(\n)
\rightarrow 
\pfin_\n
\rightarrow 
\pfin_{\Lambda,\n}
\rightarrow 
0.
\]
This splits non-canonically.

\begin{thm}
\label{thm:pfin_Lambda}
There is an isomorphism $\pfin_\bullet /\Lambda 
\cong 
(\pbar^{\otimes \bullet}/ \Lambda) 
\odot _{\fb\op}
\triv$
 of $\kring \fb\op \otimes \kring \fin$-modules and thus
\[
\pfin_\bullet 
\cong 
\pfin_{\Lambda, \bullet}
\quad \oplus \quad
(\pbar^{\otimes \bullet}/ \Lambda) 
\odot _{\fb\op}
\triv.
\]
\end{thm}

\begin{proof}
This result refines Proposition \ref{prop:filter_pfin_n}. Replacing $\pbar^{\otimes k}$ by $\pbar ^{\otimes k}/\Lambda^k (\pbar)$ yields a projective $\kring \fin$-module. Hence, the filtration obtained from Proposition \ref{prop:filter_pfin_n}  splits. The result follows.
\end{proof}

\begin{remark}
The reader is encouraged to check that this isomorphism gives an alternative approach to proving Theorem \ref{thm:decompose_Slambda_Pfin} and vice versa.
\end{remark}

\subsection{How to calculate $\hom_{\kring \fin} (\mathbb{P}_\bullet, -)$}
\label{subsect:calculate_hom_proj_cover}

Motivated by Corollary \ref{cor:multiplicity_composition_factors}, the purpose of this section is to calculate 
$\hom_{\kring \fin} (\mathbb{P}_\bullet , F)$
as a functor to $\kring \fb$-modules. 
For simplicity of exposition, we do not include the term $\mathbb{P}_{-1}$ (cf. Notation \ref{nota:calc_proj_cover}), since this  can be treated directly using Yoneda's lemma.

In the following, we use $\W(-)$ introduced in Notation \ref{nota:H_S}. 

\begin{thm}
\label{thm:hom_proj_cover}
For a $\kring \fin$-module $F$ such that $\dim F(\mathbf{t})< \infty$ for all $t \in \nat$, there is an equality in the Grothendieck group of $\kring \fb$-modules
\[
[\hom_{\kring \fin}  (\mathbb{P}_\bullet, F)]
=
[\overline{F}]\odot_\fb \W(0)
\quad + \quad
\sum_{k\geq 1}
(\sgn_k \otimes_{\sym_k} F(\mathbf{k})) \otimes \W (k-1), 
\]
where $\overline{F}$ is the sub $\kring \fin$-module of $F$ supported on non-empty finite sets.
\end{thm}

\begin{proof}
We require to calculate (up to isomorphism) the $\kring \fb$-module $\hom_{\kring \fin}  (\mathbb{P}_\bullet, F)$; evaluated on $\n$ this has value $\hom_{\kring \fin} (\mathbb{P}_n ,F)$, where the $\kring \sym_n$-action is induced by the $\kring \sym_n\op$-action on $\mathbb{P}_n$.  

For any $n \in \nat$, we have the decomposition  of $\mathbb{P}_n$ in $\kring \fin \otimes \kring \sym_n\op$-modules:
\[
\mathbb{P}_n \cong \pbar ^{\otimes n} / \Lambda^n (\pbar) \oplus \Lambda^{n+1} (\pfin) \boxtimes \sgn_n.
\]
As in Example \ref{exam:multiplicity_Lambda}, it is straightforward to calculate  $\hom_{\kring \fin} (\Lambda^{n+1}(\pfin), F)$. The key step is thus to calculate the $\kring \sym_n$-module 
$
\hom_{\kring \fin} (\pbar^{\otimes n}/\Lambda^n (\pbar) , F).
$

Theorem \ref{thm:pfin_Lambda} gives the isomorphism $
\pfin_\bullet /\Lambda 
\cong 
(\pbar^{\otimes \bullet}/ \Lambda) 
\odot_{\fb\op} 
\triv
$ of $\kring \fb\op \otimes \kring \fin$-modules and hence  
 the isomorphism
\[
\hom_{\kring \fin} (\pfin_\bullet /\Lambda , F) \cong \hom_{\kring \fin} (\pbar^{\otimes \bullet}/ \Lambda, F) \odot_\fb \triv.
\]
By  Proposition \ref{prop:invert_odot_triv}, this gives the equality in the Grothendieck group of $\kring \fb$-modules:
\[
[\hom_{\kring \fin} (\pbar^{\otimes \bullet}/ \Lambda, F) ]
= 
[\hom_{\kring \fin} (\pfin_\bullet /\Lambda , F)] \odot _\fb \W(0).
\]
Now, since $[\hom_{\kring \fin} (\pfin_\bullet, F)]= [F]$ by Yoneda's lemma, this gives 
\[
[\hom_{\kring \fin} (\pfin_\bullet /\Lambda , F)] = [F] - [\hom_{\kring \fin} (\pfin_{\Lambda, \bullet}, F)].
\]
By Proposition \ref{prop:project_Lambda_n}, if $n>0$ we have
$
\pfin_{\Lambda, \n}
\cong 
\bigoplus_{k=1} ^n 
\Lambda^k (\pfin) \boxtimes S_{(n-k+1, 1^{k-1})}$; for $n=0$,  $ \pfin_{\Lambda, \mathbf{0} } = \pfin_\mathbf{0} = \kring$. 
This gives the equality  
\[
[\hom_{\kring \fin} (\pfin_{\Lambda, \bullet}, F)]
= [F(0) \otimes \triv_0] + \sum_{n>0} \sum_{k=1}^n [\big(\sgn_k \otimes_{\sym_k}F (\mathbf{k})\big) \otimes S_{(n-k+1, 1^{k-1})}]. 
\]
Reversing the order of summation in $k$ and $n$ gives the equality 
\[
[\hom_{\kring \fin} (\pfin_{\Lambda, \bullet}, F)]
= \sum_{k \geq 0}  [\big(\sgn_k \otimes_{\sym_k}F (\mathbf{k})\big) \otimes H(k)], 
\]
where $H(k)$ is as in Notation \ref{nota:H_S} (the term $k=0$ is subsumed in the sum).

Then, combining the above and using Proposition \ref{prop:Hook_inversion} gives the equality
\[
[\hom_{\kring \fin} (\pbar^{\otimes \bullet} /\Lambda , F)]
= 
[F] \odot_\fb \W(0) 
\quad - \quad 
\sum_{k \geq 0}  \big(\sgn_k \otimes_{\sym_k}F (\mathbf{k})\big) \otimes \W(k).
\]
The term $k=0$ of the sum is equal to $F(0) \odot_\fb \W(0)$, so this can be rewritten as 
\[
[\hom_{\kring \fin} (\pbar^{\otimes \bullet} /\Lambda , F)]
= 
[\overline{F}] \odot_\fb \W(0)
\quad - \quad 
\sum_{k \geq 1}  \big(\sgn_k \otimes_{\sym_k}F (\mathbf{k})\big) \otimes \W(k).
\]

It remains to calculate $[\hom_{\kring \fin} (\mathbb{P}_\bullet , F)]$ by reintegrating the contribution from
\begin{eqnarray*}
\hom_{\kring \fin} (\Lambda^{\bullet +1 } (\pfin) , F)
&\cong  &
\bigoplus _{m\geq 0} (\sgn_{m+1} \otimes_{\sym_{m+1}} F(\mathbf{m+1}) )
\ \otimes \ \sgn_m
\\
&=& \bigoplus _{k\geq 1} (\sgn_{k} \otimes_{\sym_{k}} F(\mathbf{k}) )
\ \otimes \ \sgn_{k-1}.
\end{eqnarray*}
Using the identity $\W (k-1) =[\sgn_{k-1}] -  \W (k)$ for $k \geq 1$ (cf. Lemma \ref{lem:W_relations}), the result follows by summing the requisite expressions.
\end{proof}

\section{The composition factors of $\pfin_\n$ and $\pbar^{\otimes n}$}
\label{sect:cfactors}

In this section Theorem \ref{thm:hom_proj_cover} is applied to calculate $\hom_{\kring \fin} (\mathbb{P}_\bullet, F)$, first in the case $F=\pfin_\n$ (see Theorem \ref{thm:hom_proj_cover_pfin}), and then for $F=\pbar^{\otimes n}$ (see Theorem \ref{thm:hom_proj_cover_pbar_otimes}). 
 The first step is the calculation of  $\hom_{\kring \fin}(\Lambda^t (\pfin), \pfin_\n)$, which is achieved in Proposition \ref{prop:calculate_sgn_otimes_Pfin}. 

\subsection{Calculating $\hom (\Lambda^t (\pfin) , \pfin_\n)$}

For a fixed $t \in \nat$, we consider
$$
\hom_{\kring \fin} (\Lambda^t (\pfin) , \pfin_\n)
\cong 
\sgn_t \otimes_{\sym_t} \pfin_\n (\mathbf{t})
$$
 as a $\kring \fin\op$-module (using naturality with respect to $\n$) and hence, by restriction along $\kring \fs \op \hookrightarrow \kring \fin\op$, as a $\kring \fs\op$-module.

\begin{lemma}
\label{lem:sgn_otimes_Pfin}
There is an isomorphism of $\kring \fs\op$-modules
\[
\sgn_t \otimes_{\sym_t} \pfin_\n (\mathbf{t})
\cong 
\sgn_t \otimes_{\sym_t} \kring \finj (-, \mathbf{t}) \otimes_\fb \kring \fs (\n, -)
\]
where the $\fs\op$-naturality is with respect to $\n$.
\end{lemma}

\begin{proof}
By definition, $\pfin_\n (\mathbf{t}) = \kring \fin (\n, \mathbf{t})$. Now, 
as in Example \ref{exam:kfin}, one has the isomorphism of $\kring \fb\otimes \kring \fs\op$-modules
\[
\kring \fin (\n, \mathbf{t}) \cong \kring \finj (-, \mathbf{t}) \otimes_\fb \kring \fs (\n, -),
\]
where the $\kring \fs\op$-functoriality is with respect to $\n$. The result follows upon applying $\sgn_t \otimes_{\sym_t} -$.
\end{proof}

One has the following identification:

\begin{lemma}
\label{lem:identify_sgn_otimes_kfinj}
For $s, t \in \nat$, there is an isomorphism of $\kring \sym_s\op$-modules:
\[
\sgn_t \otimes_{\sym_t} \kring \finj (\mathbf{s}, \mathbf{t})
\cong 
\left\{
\begin{array}{ll}
\sgn_s & s \in \{t-1, t\} \cap \nat \\
0 & \mbox{otherwise.}
\end{array}
\right.
\]
\end{lemma}

\begin{proof}
Clearly $\kring \finj (\mathbf{s}, \mathbf{t})=0$ if $s>t$, so we may assume that $s \leq t$.  Now, as in  the proof of Proposition \ref{prop:identify_induction}, there is an isomorphism of $\kring (\sym_t \times \sym_s\op)$-modules $\kring \finj (\mathbf{s}, \mathbf{t}) \cong \kring \sym_t / \sym_{t-s}$, where the actions are inherited from the left and right regular actions on $\kring \sym_t$. It follows that one has the isomorphism
$ 
\sgn_t \otimes_{\sym_t} \kring \finj (\mathbf{s}, \mathbf{t})
\cong 
\sgn_t|_{\sym_{t-s}} \otimes_{\sym_{t-s}} \triv_{t-s}
$ 
with respect to the residual $\kring \sym_s$-action arising from $\sgn_t|_{\sym_s}$ on the right hand side. This is clearly zero unless $t-s \in \{0, 1\}$ in which case the representation identifies as stated.
\end{proof}

This allows the main result of this subsection to be proved:

\begin{proposition}
\label{prop:calculate_sgn_otimes_Pfin}
For $t \in \nat$, there is a short exact sequence of $\kring \fin\op$-modules:
 \[
 0 
 \rightarrow 
 \sgn_{t-1} \otimes_{\sym_{t-1}} \kring \fs (\bullet, \mathbf{t-1})
 \rightarrow 
  \sgn_t \otimes_{\sym_t} \pfin_\bullet (\mathbf{t})
\rightarrow 
 \sgn_{t} \otimes_{\sym_{t}} \kring \fs (\bullet, \mathbf{t})
 \rightarrow 
 0,
 \]
using the $\kring \fin\op$-module structure for $\kring \fs (\bullet, \mathbf{t-1})$ and $\kring \fs (\bullet, \mathbf{t})$ provided by Proposition \ref{prop:pfsop_finop-module}. (For $t=0$, the term $
 \sgn_{t-1} \otimes_{\sym_{t-1}} \kring \fs (\bullet, \mathbf{t-1})$ is understood to be zero.)
 This short exact sequence  splits after restriction to $\kring \fs\op \subset \kring \fin\op$.
\end{proposition}

\begin{proof}
The surjection $\pfin_\bullet (\mathbf{t}) \twoheadrightarrow \kring \fs (\bullet, \mathbf{t})$ of Proposition \ref{prop:pfsop_finop-module} is $\kring \sym_t \op$-equivariant, hence induces the surjection $ \sgn_t \otimes_{\sym_t} \pfin_\bullet (\mathbf{t})
\rightarrow 
 \sgn_{t} \otimes_{\sym_{t}} \kring \fs (\bullet, \mathbf{t})$ in the short exact sequence. 
It remains to identify the kernel. As a $\kring \fs\op$-module, this follows from Lemmas \ref{lem:sgn_otimes_Pfin} and \ref{lem:identify_sgn_otimes_kfinj}. One checks that the $\kring \fin\op$-module structure is as stated. 
\end{proof}

\begin{remark}
This result is related to the short exact sequence given in Corollary \ref{cor:ses_llop}.
\end{remark}

\begin{example}
 For $t=1$, $\n \mapsto \pfin_\n (\mathbf{1})$ identifies as the constant $\kring \fin\op$-module $\kring^{\fin\op}$ (the superscript indicates that this is a $\kring \fin\op$-module),  hence so does $\n \mapsto \sgn_1 \otimes_{\sym_1} \pfin_\n (\mathbf{1})$.
 The left hand term in the short exact sequence  contributes $\kring_\mathbf{0}^{\fin\op}$  and  the right hand term   $\kbar^{\fin\op}$ as $\kring \fin\op$-modules. Thus the short exact sequence  is 
\[
0
\rightarrow 
\kring_{\mathbf{0}}^{\fin\op} 
\rightarrow \kring^{\fin\op}
\rightarrow 
\kbar^{\fin\op}
\rightarrow 
0.
\]
This is not split, but splits on restriction to $\kring \fs\op$.
\end{example}

\subsection{Morphisms from $\mathbb{P}_\bullet$ to   $\pfin_\n$ and to $\pbar^{\otimes n}$}

Proposition \ref{prop:calculate_sgn_otimes_Pfin} identified $\sgn_k \otimes _{\sym_k} \pfin_\n (\mathbf{k}) $.
This gives the following consequence of Theorem \ref{thm:hom_proj_cover}:

\begin{thm}
\label{thm:hom_proj_cover_pfin}
For $n \in \nat$, there is an  equality in the Grothendieck group of $ \kring \sym_n\op \otimes   \kring \fb$-modules
\[
[\hom_{\kring \fin}  (\mathbb{P}_\bullet, \pfin_\n)]
=
[\kring \fs (\n, -)]  
\quad + \quad
\sum_{k\geq 1}
(\sgn_k \otimes_{\sym_k} \kring \fs (\n , \mathbf{k})) \otimes [\sgn_{k-1}].
\]
\end{thm}

\begin{proof}
The case $n=0$ is proved by inspection. Henceforth suppose that $n>0$.
 By Theorem \ref{thm:hom_proj_cover}, one has
\[
[\hom_{\kring \fin}  (\mathbb{P}_\bullet, \pfin_\n)]
=
[\overline{\pfin_\n}]\odot_\fb \W(0)
\quad + \quad
\sum_{k\geq 1}
(\sgn_k \otimes_{\sym_k} \pfin_n(\mathbf{k})) \otimes \W (k-1), 
\]
The hypothesis $n>0$ ensures that $\overline{\pfin_\n} = \pfin_\n$. Thus, using the identity as $\kring \fb$-bimodules 
$\kring \fin \cong \triv \odot_\fb \kring \fs \cong \kring \fs \odot _\fb \triv$ 
 given in Example \ref{exam:kring_fin_odot}, together with Proposition \ref{prop:invert_odot_triv}, one has $[\overline{\pfin_\n}]\odot_\fb \W(0) = [\kring \fs (\n, -)]$, which accounts for the first term in the statement.
 
It remains to consider the second term.  For the purposes of this proof, write
$
M(\n, \mathbf{k}) := \sgn_{k} \otimes _{\sym_{k}} \kring \fs (\n, \mathbf{k})
$ 
so that $M(-, \mathbf{k})$ is a $\kring \fb\op$-module. 
The hypothesis $n>0$ gives that $M(\n , \mathbf{0}) =0$ and clearly $M(\n, \mathbf{k})= 0$ if $k>n$.
 Using this notation, Proposition \ref{prop:calculate_sgn_otimes_Pfin} gives the isomorphism of $\kring \sym_n\op$-modules 
\[
\sgn_k \otimes _{\sym_k} \pfin_\n (\mathbf{k}) 
\cong 
M (\n, \mathbf{k-1}) \ \oplus \ M (\n, \mathbf{k}).
\]
Hence 
\begin{eqnarray*}
\sum_{k\geq 1}
(\sgn_k \otimes_{\sym_k} \pfin_n(\mathbf{k})) \otimes \W (k-1)
&=&
\sum_{k \geq 1} 
M (\n, \mathbf{k}) \otimes (\W (k-1) \oplus \W (k) )
\\
\mbox{\tiny{(Lemma \ref{lem:W_relations})}}
&=&
\sum_{k \geq 1} 
M (\n, \mathbf{k}) \otimes [\sgn_{k-1}],
\end{eqnarray*}
where the first equality reindexes the sum. The result follows.
\end{proof}

The conclusion of Theorem \ref{thm:hom_proj_cover} only depends upon the underlying $\kring \fb$-module of the $\kring \fin$-module considered. This allows us to apply Proposition \ref{prop:filter_pfin_n}, which has the following consequence:

\begin{lemma}
\label{lem:pfin_versus_pbar_otimes}
There is an isomorphism $\overline{\pfin_\bullet}
\cong 
\pbar^{\otimes \bullet} \odot_{\fb\op} \triv$ as 
 $\kring \fb$-bimodules, 
where $\overline{\pfin_\n} = \pfin_\n$ unless $n=0$. 
\end{lemma}

Hence, using Proposition \ref{prop:invert_odot_triv}, one can calculate $[\hom_{\kring \fin}  (\mathbb{P}_\bullet, \pbar^{\otimes *})]$:

\begin{thm}
\label{thm:hom_proj_cover_pbar_otimes}
There is an equality in the Grothendieck group of $\kring \fb$-bimodules
\begin{eqnarray*}
[\hom_{\kring \fin}  (\mathbb{P}_\bullet, \pbar^{\otimes *})]
&=&
[\kring \fs ]\odot_{\fb\op} \W(0) 
\\
&&  + \ 
\sum_{k\geq 1}
\Big([\sgn_k \otimes_{\sym_k}\kring \fs(-, \mathbf{k})]\odot_{ \fb\op} \W(0) \Big)  \boxtimes [\sgn_{k-1}]_\fb,
\end{eqnarray*}
where the place-holder $\bullet$ corresponds to the $\kring\fb$-module structure and $*$ to the $\kring \fb\op$-module structure.
\end{thm}

\begin{proof}
By Lemma \ref{lem:pfin_versus_pbar_otimes} together with Proposition \ref{prop:invert_odot_triv}, one has the equality
\[
[\hom_{\kring \fin} (\mathbb{P}_\bullet, \pbar^{\otimes *})]
= 
 [\hom_{\kring \fin} (\mathbb{P}_\bullet,\overline{\pfin_*})]\odot_{\fb\op} \W(0) .
\]
The result follows by using the identification of  $[\hom_{\kring \fin} (\mathbb{P}_\bullet,\overline{\pfin_*})]$  given by Theorem \ref{thm:hom_proj_cover_pfin}.
\end{proof}

Theorem \ref{thm:hom_proj_cover_pbar_otimes} has the following consequence:

\begin{corollary}
\label{cor:endo_proj_cover}
There is an equality in the Grothendieck group of $\kring \fb$-bimodules:
\begin{eqnarray*}
[\hom_{\kring \fin}  (\mathbb{P}_\bullet, \mathbb{P}_*)]
&=&
[\hom_{\kring \fin}  (\mathbb{P}_\bullet, \pbar^{\otimes *} )]
\ + \  
\sum_{\ell \geq 0}
[\sgn_{\ell}]_{\fb\op} \boxtimes [\sgn_{\ell +1}]_\fb
\\
&=&
[\kring \fs ]\odot_{\fb\op} \W(0) 
\\
&& + \quad 
\sum_{k\geq 1}
\big([\sgn_k \otimes_{\sym_k}\kring \fs(-, \mathbf{k})]\odot_{ \fb\op} \W(0) \big)  \boxtimes [\sgn_{k-1}]_\fb
\\
&&
+ \quad
 \sum_{\ell \geq 0}
[\sgn_{\ell}]_{\fb\op} \boxtimes [\sgn_{\ell +1}]_\fb.
\end{eqnarray*}
\end{corollary}

\begin{proof}
The result follows by using the short exact sequences  for $\ell \geq 0$
\[
0
\rightarrow 
\Lambda^{\ell +1 }(\pbar) \boxtimes \sgn_{\ell}
\rightarrow 
\mathbb{P}_{\ell}
\rightarrow 
\pbar^{\otimes \ell}
\rightarrow 
0,
\]
keeping track of the $\sym_\ell\op$-action, as indicated.
 Clearly $\hom_{\kring \fin}(\mathbb{P}_\bullet, \Lambda^{\ell +1 }(\pbar))$ is non-zero if and only if $\bullet = \ell +1$. This accounts for the additional terms as compared with Theorem \ref{thm:hom_proj_cover_pbar_otimes}.
\end{proof}

\section{Calculating $\hom (\pbar^{\otimes \bullet}, \pbar^{\otimes *})$}
\label{sect:calculate}

The purpose of this section is to calculate the underlying $\kring \fb$-bimodule of 
$$
\hom_{\kring \fin} (\pbar^{\otimes \bullet}, \pbar^{\otimes *}).
$$ 
 Theorem \ref{thm:hom_proj_cover_pbar_otimes} described $[\hom_{\kring \fin}  (\mathbb{P}_\bullet, \pbar^{\otimes *})]$ in the Grothendieck group of $\kring \fb$-bimodules. We exploit the short exact sequence 
\[
0
\rightarrow 
\Lambda^{\bullet +1} (\pbar) 
\rightarrow 
\mathbb{P}_\bullet 
\rightarrow 
\pbar^{\otimes \bullet}
\rightarrow 
0.
\]

We first consider the calculation of   $\hom_{\kring \fin} (\Lambda^{\bullet +1} (\pbar), -)$;  a general result is given in Proposition \ref{prop:complex_Lambda_description} and the case of interest for the applications in Proposition \ref{prop:hom_Lambda_Pfin}. Then these results are put together so as to prove the main result, Theorem \ref{thm:endo_pbar_otimes}.

\subsection{Morphisms out of $\Lambda^s (\pbar)$}
\label{subsect:exterior}

For any $s \in \nat$, $\Lambda^s (\pbar)$ is a simple functor (by convention, for $s=0$, this is $\kbar$) and it has a projective presentation 
\[
\Lambda^{s+2} (\pfin) \rightarrow \Lambda^{s+1} (\pfin) 
\rightarrow 
\Lambda^s (\pbar) 
\rightarrow 
0,
\]
provided by the truncation of the exact complex of Proposition \ref{prop:Lambda-complex}. Thus, for 
 a $\kring \fin$-module $F$, there is a natural exact sequence
\begin{eqnarray}
\label{eqn:copresent_hom_Lambda}
0
\rightarrow 
\hom_{\kring \fin} (\Lambda^s (\pbar) , F) 
\rightarrow 
\hom_{\kring \fin} (\Lambda^{s+1} (\pfin) , F) 
\rightarrow 
\hom_{\kring \fin} (\Lambda^{s+2} (\pfin) , F) .
\end{eqnarray}

Fix $t \in \nat$. 
For  a $\kring \fin$-module $F$, one has the natural isomorphism (cf. Example \ref{exam:multiplicity_Lambda})
\begin{eqnarray}
\label{eqn:hom_Lambda}
\hom_{\kring \fin} (\Lambda ^t (\pfin), F) \cong \sgn_t \otimes_{\sym_t} F(\mathbf{t}).
\end{eqnarray}
We seek to identify the right hand map in (\ref{eqn:copresent_hom_Lambda}) in terms of the respective natural isomorphisms (\ref{eqn:hom_Lambda}).

There is a bijection 
$ 
\finj (\mathbf{t}, \mathbf{t+1}) \cong \aut (\mathbf{t+1})
$, 
since an injection from $\mathbf{t}$ to $\mathbf{t+1}$ extends uniquely to an automorphism and all such arise in this way. Hence, for $i \in \finj (\mathbf{t}, \mathbf{t+1})$, one can define  $\sgn(i)\in \{ \pm 1 \}$ as that of the corresponding automorphism of $\mathbf{t+1}$. This is used in the following:

\begin{definition}
\label{defn:sigma}
Let 
$ 
\sigma_\mathbf{t} \in \kring \finj (\mathbf{t}, \mathbf{t+1}) 
$ 
be 
\begin{eqnarray}
\label{eqn:sigma_t}
\sigma_\mathbf{t}:= \sum_{i \in \finj^{\mathrm{or}} (\mathbf{t}, \mathbf{t+1})} \sgn(i) [i],
\end{eqnarray}
 where  $\finj^{\mathrm{or}} (\mathbf{t}, \mathbf{t+1})$ is the set of order-preserving inclusions. 

Let $\overline{\sigma_\mathbf{t}}$ be the image of $\sigma_\mathbf{t}$ under the projection:
\[
\kring \finj (\mathbf{t}, \mathbf{t+1}) 
\twoheadrightarrow 
\sgn_{t+1} \otimes_{\sym_{t+1}} \kring \finj (\mathbf{t}, \mathbf{t+1}) \otimes_{\sym_t} \sgn_t \cong \kring.
\]
\end{definition}

\begin{lemma}
\label{lem:identify_Hom_diff}
For $t \in \nat$ and a $\kring \fin$-module $F$, the right hand  map of (\ref{eqn:copresent_hom_Lambda}) identifies via the respective isomorphisms (\ref{eqn:hom_Lambda}) with the natural transformation induced by $\overline{\sigma_\mathbf{t} }$:
\[
\sgn_t \otimes_{\sym_t} F (\mathbf{t})
\rightarrow
 \sgn_{t+1} \otimes_{\sym_{t+1}}F( \mathbf{t+1}).
\]
\end{lemma}

\begin{proof}
The differential $\Lambda^{t+1} (\pfin) \rightarrow \Lambda^t (\pfin)$ identifies as  the composite
\[
\Lambda^{t+1} (\pfin) \rightarrow \Lambda^t (\pfin) \otimes \pfin 
\rightarrow 
\Lambda^t (\pfin), 
\]
where the first map is induced by the coproduct $\Lambda^{t+1} \rightarrow \Lambda^t \otimes \Lambda^1$ and the second is given by $\pfin \rightarrow \kring$ on the second factor. Unwinding the definitions, one checks that this yields the map induced by $\overline{\sigma_\mathbf{t}}$.
\end{proof}

\begin{proposition}
\label{prop:complex_Lambda_description}
For $s \in \nat$ and a $\kring \fin$-module $F$, the truncation of the complex of Proposition \ref{prop:Lambda-complex} yields a natural  exact complex:
\[
0
\rightarrow 
\hom_{\kring \fin} (\Lambda^s (\pbar) , F) 
\rightarrow 
\sgn_{s+1} \otimes_{\sym_{s+1}} F (\mathbf{s+1})
\stackrel{F(\overline{\sigma_\mathbf{s+1}})}{\longrightarrow}
 \sgn_{s+2} \otimes_{\sym_{s+2}}F( \mathbf{s+2}).
\]
\end{proposition}

Taking $F$ to be a standard projective, the above  can be made entirely explicit:

\begin{proposition}
\label{prop:hom_Lambda_Pfin}
For $F = \pfin_\bullet$ and $s \in \nat$, there is an isomorphism of $\kring \fin\op$-modules:
\[
\hom_{\kring \fin} (\Lambda^s (\pbar) , \pfin_\bullet) 
\cong 
\sgn_{s} \otimes_{\sym_{s}} \kring \fs (\bullet, \mathbf{s}),
\]
using the $\kring \fin\op$-module structure of Proposition \ref{prop:pfsop_finop-module} on $\kring \fs(\bullet, \mathbf{s})$.
 Moreover, with respect to this isomorphism, for $s,t \in \nat$, the short exact sequence $0 \rightarrow \Lambda^{s+1}(\pbar) \rightarrow \Lambda^{s+1}(\pfin) \rightarrow \Lambda^s (\pbar) \rightarrow 0$ induces a short exact sequence  (writing $\hom$ for $\hom_{\kring \fin}$)
\[
0
\rightarrow 
\hom(\Lambda^s (\pbar), \pfin_\mathbf{t}) 
\rightarrow 
\hom(\Lambda^{s+1}(\pfin), \pfin_\mathbf{t}) 
\rightarrow 
\hom(\Lambda^{s+1} (\pbar), \pfin_\mathbf{t}) 
\rightarrow 
0
\]
and this identifies with the short exact sequence of Proposition \ref{prop:calculate_sgn_otimes_Pfin}.
\end{proposition}

\begin{proof}
The first statement is an immediate consequence of the natural isomorphism 
 $$
\hom_{\kring \fin} (\pbar^{\otimes s}, \pfin_\bullet)
\cong 
\kring \fs (\bullet, \mathbf{s})
$$ 
provided by Theorem \ref{thm:morphism_pbar_otimes_n_Pfin}. Moreover, Corollary \ref{cor:model_right_augmentation}  implies that the inclusion $\Lambda^{s+1} (\pbar) \hookrightarrow \Lambda^{s+1} (\pfin_\mathbf{1})$ induces a surjection 
\[
\hom_{\kring \fin}(\Lambda^{s+1}(\pfin), \pfin_\bullet) 
\twoheadrightarrow 
\hom_{\kring \fin}(\Lambda^{s+1} (\pbar), \pfin_\bullet).
\]
It follows that one has a short exact sequence of the given form and the surjection corresponds to 
\[
\sgn_{s+1} \otimes_{\sym_{s+1}} \kring \fin (\bullet, \mathbf{s+1})
\twoheadrightarrow
\sgn_{s+1} \otimes_{\sym_{s+1}} \kring \fs (\bullet, \mathbf{s+1})
\]
appearing in the short exact sequence of  Proposition \ref{prop:calculate_sgn_otimes_Pfin}. The result follows.
\end{proof}

\subsection{Calculating $\hom (\Lambda^s (\pbar), \pbar^{\otimes *})$}

In order to calculate $\hom_{\kring \fin} (\Lambda^s (\pbar), \pbar^{\otimes *})$, we first treat the exceptional factors $\pfin_{\Lambda, \bullet}$, using the Hook representation $\kring \fb\op$-modules  of Notation \ref{nota:H_S}:

\begin{lemma}
\label{lem:hom_Lambda_pbar_pfin_Lambda}
For $s \in \nat$, there is an isomorphism of $\kring \fb\op$-modules
\[
\hom_{\kring \fin} (\Lambda^s (\pbar) , \pfin_{\Lambda,\bullet}) 
\cong 
H(s).
\]
\end{lemma}

\begin{proof}
By Proposition \ref{prop:project_Lambda_n}, 
 for $n>0$, 
$
\pfin_{\Lambda, \n}
\cong 
\bigoplus_{k=1} ^n 
\Lambda^k (\pfin) \boxtimes S_{(n-k+1, 1^{k-1})}$ and, for $n=0$, $\pfin_{\Lambda, \mathbf{0}}= \kring$ and the latter is, by convention,  $\Lambda^0 (\pfin)$. 
 Now,  $\hom_{\kring \fin} (\Lambda^s (\pbar) , \Lambda^k (\pfin))$ is zero unless $s=k$, when it is $\kring$, by the results of Section \ref{subsect:structure_Lambda_pfin}. It follows, that 
\[
\hom_{\kring \fin} (\Lambda^s (\pbar) , \pfin_{\Lambda,\n})
=
\left\{ 
\begin{array}{ll}
\kring & s=0=n \\
S_{(n-s+1, 1^{s-1})} & n \geq s > 0 
\\
0 &\mbox{ otherwise.}
\end{array}
\right. 
\] 
The result follows, by the definition of $H(s)$.
\end{proof}

One deduces:

\begin{proposition}
\label{prop:hom_Lambda_pbar_pbar_otimes/Lambda}
For $s \in \nat$, there are isomorphisms of $\kring \fb\op$-modules 
\begin{eqnarray*}
\hom_{\kring \fin}(\Lambda^s (\pbar), \pbar^{\otimes \bullet}/\Lambda) \odot_{\fb\op} \triv 
\  \oplus \  
H(s) 
&\cong & 
\hom_{\kring \fin}(\Lambda^s (\pbar),\pfin_\bullet)
\\
&\cong &
\sgn_{s} \otimes_{\sym_{s}} \kring \fs (\bullet, \mathbf{s})
.
\end{eqnarray*}
Hence there are equalities in the Grothendieck group of $\kring \fb\op$-modules:
\begin{eqnarray*}
[\hom_{\kring \fin}(\Lambda^s (\pbar), \pbar^{\otimes *}/\Lambda)] 
&=&
[\sgn_{s} \otimes_{\sym_{s}} \kring \fs (-, \mathbf{s})]\odot_{\fb\op}\W(0) 
\quad - \quad \W (s)
\\ 
\ 
[\hom_{\kring \fin}(\Lambda^s (\pbar), \pbar^{\otimes *}] 
&=&
[\sgn_{s} \otimes_{\sym_{s}} \kring \fs (-, \mathbf{s})]\odot_{\fb\op}\W(0) 
\quad + \quad \W (s+1).
\end{eqnarray*}
\end{proposition}

\begin{proof}
By Theorem \ref{thm:pfin_Lambda}, 
$
\pfin_\bullet /\Lambda 
\cong 
(\pbar^{\otimes \bullet}/ \Lambda) 
\odot_{\fb\op} 
\triv
$, which gives: 
\begin{eqnarray*}
\hom_{\kring \fin}(\Lambda^s (\pbar), (\pbar^{\otimes *}/\Lambda) \odot_{\fb\op} \triv) 
&\cong & 
\hom_{\kring \fin}(\Lambda^s (\pbar), \pbar^{\otimes *}/\Lambda) \odot_{\fb\op} \triv 
\\
&\cong & 
\hom_{\kring \fin}(\Lambda^s (\pbar),\pfin_\bullet /\Lambda ).
\end{eqnarray*}
Using the decomposition $\pfin_\bullet \cong \pfin_\bullet/\Lambda \ \oplus \  \pfin_{\Lambda, \bullet}$, the first statement follows by combining Proposition  \ref{prop:hom_Lambda_Pfin}  with Lemma \ref{lem:hom_Lambda_pbar_pfin_Lambda}. 
 The statement for $[\hom_{\kring \fin}(\Lambda^s (\pbar), \pbar^{\otimes *}/\Lambda)]$ 
 then follows by inverting $-\odot_{\fb\op} \triv$ using Proposition \ref{prop:invert_odot_triv};  Proposition \ref{prop:Hook_inversion} is used to treat the contribution from the hook representations. 
 Finally, the only difference when passing from $\pbar^{\otimes *}/\Lambda$ to $\pbar^{\otimes *}$ is a copy of $[\sgn_s]$ from the direct summand $\Lambda^s (\pbar) $ of $\pbar^{\otimes s}$. The final statement  follows from the equality $[\sgn_s] - \W (s) = \W (s+1)$ given by Lemma \ref{lem:W_relations}. 
\end{proof}

\subsection{The calculation of $\hom (\pbar^{\otimes \bullet}, \pbar^{\otimes *})$}

Recall that we have the short exact sequence 
\[
0
\rightarrow 
\Lambda^{\bullet + 1} (\pbar) 
\rightarrow 
\mathbb{P}_\bullet
\rightarrow 
\pbar^{\otimes \bullet } 
\rightarrow 
0.
\]
This should be considered as a short exact sequence of $\kring \fb\op \otimes \kring \fin$-modules. To do so, the term $\Lambda^{\bullet +1} (\pbar)$ must be understood as being 
\[
\bigoplus _{k \geq 0} \Lambda^{k+1} (\pbar) \boxtimes \sgn_k.
\]
 The difference between $\mathbb{P}_\bullet$ and $\pbar^{\otimes \bullet } $ is encoded in the short exact sequences
\begin{eqnarray}
\label{eqn:ses_Lambda_bullet_k}
0
\rightarrow 
\Lambda^{k + 1} (\pbar)  \boxtimes \sgn_k
\rightarrow 
\Lambda^{k+1} (\pfin)\boxtimes \sgn_k
\rightarrow 
\Lambda^k (\pbar) \boxtimes \sgn_k
\rightarrow 
0.
\end{eqnarray}

We first consider applying  $\hom_{\kring \fin}(-, \pbar^{\otimes *}/\Lambda)$:

\begin{lemma}
\label{lem:ses_pbar/Lambda}
Applying the functor $\hom_{\kring \fin}(-, \pbar^{\otimes *}/\Lambda)$ to the short exact sequences (\ref{eqn:ses_Lambda_bullet_k}) yields the short exact sequence of $\kring \fb\op$-bimodules (writing $\hom$ for $\hom_{\kring\fin}$) 
\[
0
\rightarrow 
\hom(\Lambda^\bullet (\pbar) , \pbar^{\otimes *}/\Lambda)
\rightarrow 
\hom(\Lambda^{\bullet +1} (\pfin), \pbar^{\otimes *}/\Lambda)
\rightarrow 
\hom(\Lambda^{\bullet + 1} (\pbar) , \pbar^{\otimes *}/\Lambda)
\rightarrow 
0,
\]
where for $\bullet = k$, the group $\sym_k\op$ acts by the sign representation $\sgn_k$.
\end{lemma}

\begin{proof}
The short exact sequence is deduced from that of Proposition  \ref{prop:hom_Lambda_Pfin}, using the fact that $\pbar^{\otimes n} /\Lambda^n (\pbar)$ is a direct summand of $\pfin_\n$. The $\kring \fb\op$-action follows from that given in (\ref{eqn:ses_Lambda_bullet_k}).
\end{proof}

By construction of $\mathbb{P}_\bullet$, one has an isomorphism
$$
\hom_{\kring \fin} (\mathbb{P}_\bullet,\Lambda^* (\pbar) ) \cong \hom_{\kring \fin} (\pbar^{\otimes \bullet},\Lambda^* (\pbar)).
$$
This combined with Lemma \ref{lem:ses_pbar/Lambda} gives:
 	  
\begin{lemma}
\label{lem:relate_hom_projcover_hom_pbar_otimes}
There is an equality in the Grothendieck group of $\kring \fb$-bimodules:
\begin{eqnarray*}
[\hom_{\kring \fin} (\mathbb{P}_\bullet, \pbar^{\otimes *})]
&=&
[\hom_{\kring \fin} (\pbar^{\otimes \bullet}, \pbar^{\otimes *})]
\\
&& + 
\sum_{k \geq 1} [\hom_{\kring \fin}(\Lambda^k (\pfin), \pbar^{\otimes *}/\Lambda )] \boxtimes [\sgn_{k-1}]_{\fb},
\end{eqnarray*}
in which  $[\sgn_{k-1}]_\fb$ indicates that $\sgn_{k-1}$ is considered as a $\kring \fb$-module.
\end{lemma}

\begin{proof}
By the above discussion, Lemma  \ref{lem:ses_pbar/Lambda} implies that, on applying the functor $\hom_{\kring \fin} (-, \pbar^{\otimes *})$ to the short exact sequences
(\ref{eqn:ses_Lambda_bullet_k}), one obtains the short exact sequence of $\kring \fb$-bimodules (writing $\hom$ for $\hom_{\kring \fin}$)
\[
0
\rightarrow 
\hom (\pbar^{\otimes \bullet} , \pbar^{\otimes *})
\rightarrow 
\hom (\mathbb{P}_\bullet , \pbar^{\otimes *})
\rightarrow 
\bigoplus_{k \geq 1}
 \hom (\Lambda^k (\pbar) \boxtimes \sgn_{k-1}, \pbar^{\otimes *}/\Lambda)
\rightarrow 
0.
\]
The result follows on passing to isomorphism classes of $\kring \fb$-bimodules.
\end{proof}

Finally, we arrive at:

\begin{thm}
\label{thm:endo_pbar_otimes}
There is an  equality in the Grothendieck group of $\kring \fb$-bimodules:
\begin{eqnarray*}
[\hom_{\kring \fin}  (\pbar^{\otimes \bullet}, \pbar^{\otimes *})]
&=&
[\kring \fs ]\odot_{\fb\op} \W(0) 
 + 
\sum_{k \geq 1} \W (k) \boxtimes [\sgn_{k-1}]_\fb
\\
&=&
[\kring \fs ]\odot_{\fb\op} \W(0) 
 + 
\sum_{k \geq 1}\sum_{t \in \nat} (-1)^t [\sgn_{k+t}]_{\fb\op}\boxtimes [\sgn_{k-1}]_\fb.
\end{eqnarray*}
\end{thm}

\begin{proof}
Proposition \ref{prop:hom_Lambda_pbar_pbar_otimes/Lambda} gives
\begin{eqnarray*}
 [\hom_{\kring \fin}(\Lambda^k (\pbar), \pbar^{\otimes *}/\Lambda)] 
&=&
\big([\sgn_{k} \otimes_{\sym_{k}} \kring \fs (-, \mathbf{k})]\odot_{\fb\op}\W(0) \big) \boxtimes [\sgn_{k-1}]_\fb
\\
&&
\  - \  \W(k) \boxtimes [\sgn_{k-1}]_\fb,
\end{eqnarray*}
in which the terms $\W (-)$ belong to the Grothendieck group of $\kring \fb\op$-modules. Also, by Theorem \ref{thm:hom_proj_cover_pbar_otimes}, we have
\begin{eqnarray*}
[\hom_{\kring \fin}  (\mathbb{P}_\bullet, \pbar^{\otimes *})]
&=&
[\kring \fs ]\odot_{\fb\op} \W(0) 
\\ &&  + \ 
\sum_{k\geq 1}
\Big([\sgn_k \otimes_{\sym_k}\kring \fs(-, \mathbf{k})]\odot_{ \fb\op} \W(0) \Big)  \boxtimes [\sgn_{k-1}]_\fb.
\end{eqnarray*}
The first equality  thus follows from Lemma  \ref{lem:relate_hom_projcover_hom_pbar_otimes} by taking the difference of these two expressions.
 The second equality  follows by developing the expression $\W(k)$.
\end{proof}

\begin{remark}
A result equivalent to Theorem \ref{thm:endo_pbar_otimes} was first proved in \cite{MR4518761} by very different methods. An alternative approach is also given in \cite{2024arXiv240711627P}.
\end{remark}


\providecommand{\bysame}{\leavevmode\hbox to3em{\hrulefill}\thinspace}
\providecommand{\MR}{\relax\ifhmode\unskip\space\fi MR }
\providecommand{\MRhref}[2]{%
  \href{http://www.ams.org/mathscinet-getitem?mr=#1}{#2}
}
\providecommand{\href}[2]{#2}

\end{document}